\documentclass[12pt,twoside]{amsart}
\usepackage{amssymb,amscd}
\usepackage{mathrsfs}
\usepackage{array,float}
\usepackage[all]{xy}
\usepackage{enumerate}
\usepackage{xcolor}

\theoremstyle{plain}
\newtheorem{thm}{Theorem}[section]

\newtheorem{lem}[thm]{Lemma}
\newtheorem{pro}[thm]{Proposition}
\newtheorem{cor}[thm]{Corollary}

\newtheorem{problem}{Problem}

\theoremstyle{remark}
\newtheorem{rem}[thm]{Remark}

\theoremstyle{definition} 
\newtheorem{definition}[thm]{Definition}

\newcommand{\Z}{\mathbb{Z}}   
\newcommand{\Q}{\mathbb{Q}}
\newcommand{\N}{\mathbb{N}}
\newcommand{\F}{\mathbb{F}}

\newcommand{\ackn}{  \noindent{\sc Acknowledgement }\hspace{5pt} }

\renewcommand{\phi}{\varphi}

\begin{document}

\author{Ilir Snopce}
\address{Universidade Federal do Rio de Janeiro\\
  Instituto de Matem\'atica \\
  21941-909 Rio de Janeiro, RJ \\ Brasil }
\email{ilir@im.ufrj.br}

\title[Frattini-injective pro-$p$ groups] {Frattini-injectivity and Maximal pro-$p$ Galois groups}

\author{Slobodan Tanushevski} \address{Universidade Federal Fluminense\\
  Instituto de Matem\'atica e Estat\'istica\\
  24210-201  Nit\'eroi, RJ \\ Brasil }

\email{stanushevski@id.uff.br}

\begin{abstract} 
We call a pro-$p$ group $G$ Frattini-injective if distinct finitely generated subgroups of $G$ have distinct Frattinis.
This paper is an initial effort toward a systematic study of Frattini-injective pro-$p$ groups (and several other related concepts). 
Most notably, we classify the $p$-adic analytic and the solvable Frattini-injective pro-$p$ groups, and 
we describe the lattice of normal abelian subgroups of a Frattini-injective pro-$p$ group. 

We prove that every maximal pro-$p$ Galois group of a field that contains a primitive $p$th root of unity (and also contains $\sqrt{-1}$ if $p=2$) is Frattini-injective.
In addition, we show that many substantial results on maximal pro-$p$ Galois groups are in fact consequences of Frattini-injectivity.
For instance, a $p$-adic analytic or solvable pro-$p$ group is Frattini-injective if and only if it can be realized as a  
maximal pro-$p$ Galois group of a field that contains a primitive $p$th root of unity (and also contains $\sqrt{-1}$ if $p=2$);
and every Frattini-injective pro-$p$ group contains a unique maximal abelian normal subgroup.

\end{abstract}

\let\thefootnote\relax\footnotetext{\textit{Mathematics Subject Classification (2020): }
Primary  20E18, 12F10; Secondary 12G05, 22E20.}


\maketitle

\section{Introduction}

Throughout, $p$ stands for a prime number.
Given a pro-$p$ group $G$, we denote by $\Phi(G)$ the Frattini subgroup of $G$.

\begin{definition}
We say that a pro-$p$ group $G$ is \emph{Frattini-injective} if the function ${H \mapsto \Phi(H)}$,  from the set of finitely generated subgroups of $G$ into itself, is injective. 
\end{definition}

This paper is an initial effort toward a systematic study of Frattini-injective pro-$p$ groups  (and several other related concepts).
Straightforward examples of Frattini-injective pro-$p$ groups are provided by the free abelian pro-$p$ groups. 
In fact, they are the only Frattini-injective abelian pro-$p$ groups,
since all Frattini-injective pro-$p$ groups are torsion-free or, better yet, they all have the unique extraction of roots property (see Corollary~\ref{virtually-abelian}).

Every subgroup of a Frattini-injective pro-$p$ group is also Frattini-injective. 
Furthermore, if $U$ is an open subgroup of a Frattini-injective pro-$p$ group $G$, then $d(U) \geq d(G)$ (see Proposition~\ref{monotone}).
This is already an indication that Frattini-injectivity is a quite restrictive condition. Nonetheless, as we shall soon see, 
many significant pro-$p$ groups are indeed Frattini-injective. 

\medskip

Our first substantial result is 

\begin{thm}\label{main}
Let $G$ be a $p$-adic analytic pro-$p$ group of dimension $d \geq 1$. Then, $G$ is Frattini-injective  if and only if it is isomorphic to one of the following groups:
 \begin{enumerate}
  \item the abelian group $\Z_p^d$;
  \item  the metabelian group $\langle x \rangle \ltimes \Z_p^{d-1}$, where $\langle x \rangle \cong \Z_p$
 and $x$ acts on $\Z_p^{d-1}$ as scalar multiplication by $\lambda$, with $\lambda = 1+p^s$ for some $s \geq 1$ if $p>2$,
 and $\lambda = 1 + 2^s$ for some $s \geq 2$ if $p=2$.
 \end{enumerate}
\end{thm}

It turns out that the Frattini-injective solvable pro-$p$ groups are also quite scarce. 

\begin{thm}
\label{solvable}
Let $G$ be a solvable pro-$p$ group. Then, $G$ is Frattini-injective if and only if  it is free abelian or isomorphic
to a semidirect product $\langle x \rangle \ltimes A$, where $\langle x \rangle \cong \Z_p$, $A$ is a free abelian pro-$p$ group and 
$x$ acts on $A$ as scalar multiplication by $1+p^s$, with $s \geq 1$ if $p$ is odd, and $s \geq 2$ if $p=2$.
\end{thm}

Next, we give a complete description of the lattice of  abelian normal subgroups of a Frattini-injective pro-$p$ group. 

\begin{thm}
\label{max-normal}
Let $G$ be a Frattini-injective pro-$p$ group. Then, $G$ has a unique maximal normal abelian subgroup $N$.
Moreover, the following assertions hold:
\begin{enumerate}[(i)]
\item $N$ is isolated in $G$.
\item Every subgroup of $N$ is normal in $G$.
\item If $Z(G) \neq 1$, then $N=Z(G)$.
\item If $Z(G)=1$ but $N \neq 1$, then $G \cong \Z_p \ltimes C_G(N)$ and $Z(C_G(N))=N$.
\end{enumerate}
\end{thm}

There are two obvious ways of sharpening the Frattini-injectivity condition. Instead of confining to finitely generated subgroups, one can take all subgroups into consideration: 
We call a pro-$p$ group $G$ \emph{strongly Frattini-injective} if the function ${H \mapsto \Phi(H)}$,  from the set of all subgroups of $G$ into itself, is injective.
Alternatively, we may require the map ${H \mapsto \Phi(H)}$ to be an embedding of posets:   
A pro-$p$ group $G$ is defined to be \emph{strongly Frattini-resistant} (\emph{Frattini-resistant}) if for all (finitely generated) subgroups $H$ and $K$ of $G$,
\[ H \leq K \iff \Phi(H) \leq \Phi(K). \]
(To understand our reason for the choice of terms,  see Section~\ref{Hierarchical triples}, where Frattini-resistance and another related concept, commutator-resistance, are introduced in a unified manner.)   

Every (strongly) Frattini-resistant pro-$p$ group is (strongly) Frattini-injective.
All solvable and all $p$-adic analytic Frattini-injective pro-$p$ groups are strongly Frattini-resistant.
Additional examples of strongly Frattini-resistant groups are provided by the free pro-$p$ groups (Theorem~\ref{free}).

If $G$ is a Demushkin group, then $G^{ab}\cong \mathbb{Z}_p^d$ or $G^{ab}\cong \mathbb{Z}/p^e\mathbb{Z} \times \mathbb{Z}_p^{d-1}$ for some $e \geq 1$; 
set $q(G):=p^e$ in the latter and $q(G):=0$ in the former case. 

\begin{thm}
\label{Dem-Frattini-resistant-free}
Let $G$ be a Demushkin pro-$p$ group. Then, the following assertions hold: 
\begin{enumerate}[(i)]
\item If $q(G) \neq p$, or $q(G)=p$ and $p$ is odd, then $G$ is strongly Frattini-resistant.
\item If $q(G)=2$ and $d(G)>2$, then $G$ is Frattini-injective, but not Frattini-resistant.  
\item If $q(G)=2$ and $d(G)=2$, then $G$ is not Frattini-injective.
\end{enumerate}
\end{thm}

The absolute Galois group of a field $k$ is the profinite group $G_k=: \textrm{Gal}(k_s/k)$, where $ k_s$ is a separable closure of $k$. 
The maximal pro-$p$ Galois group of $k$, denoted by $G_k(p)$, is the maximal pro-$p$ quotient of $G_k$.  Equivalently, $G_k(p) =  \textrm{Gal}(k(p)/k)$, where  $k(p)$ is 
the compositum of all finite Galois $p$-extensions of $k$ inside $k_s$. Delineating absolute (maximal pro-$p$) Galois groups of fields within the category of profinite (pro-$p$) groups is one of the central problems of Galois theory.

\begin{thm}
 Let $k$ be a field containing a primitive $p$th root of unity.  If $p=2$, in addition, assume that $ \sqrt{-1}  \in k$. Then $G_k(p)$ is strongly Frattini-resistant. 
\end{thm}

\begin{thm}
\label{absolute Galois}
For any field $k$ and odd prime $p$, every pro-$p$ subgroup of the absolute Galois group $G_k$ is strongly Frattini-resistant.
Moreover, if $ \sqrt{-1}  \in k$, then also every pro-$2$ subgroup of $G_k$ is strongly Frattini-resistant.
\end{thm}


In what follows, $k$ is a field containing a primitive $p$th root of unity, and also $\sqrt{-1} \in k$ if $p=2$.
In the last few decades, substantial progress has been made in the direction of finding necessary conditions for a pro-$p$ group to be realizable as the maximal pro-$p$ Galois group of some field $k$ 
(cf. for example,  \cite{BeNiJaJo07}, \cite{CheEfMi12}, \cite{Ef97}, \cite{Ef}, \cite{EfMi11}, \cite{EfMa17}, \cite{EfQu19}, \cite{Ko98},  \cite{Ko01},  \cite{MiTa16},   \cite{MiTa17}, \cite{QuWe18},  \cite{Vo11}  and references therein). 
Most notably, it follows from the positive solution of the Bloch-Kato conjecture by Rost and Voevodsky (with a `patch' of Weibel; cf.  \cite{Ro02},  \cite{Vo11} and  \cite{We09}) that every maximal pro-$p$ Galois group $G_k(p)$ is quadratic, i.e.,
the cohomology algebra $H^\bullet(G_k(p),\F_p) = \bigoplus_{n\geq0}H^n(G_k(p),\F_p)$ is generated by elements of degree $1$ and defined by homogeneous relations of degree $2$ 
(see \cite{QuSnVa19}).  
Another restriction discovered recently concerns the external cohomological structure of $G_k(p)$, more precisely, for every
$\varphi_1, \varphi_2, \varphi_3 \in H^1(G_K(p),  \mathbb{F}_p)$ the triple Massey product $ \langle \varphi_1, \varphi_2, \varphi_3  \rangle$ is not essential (cf. \cite{EfMa17}, \cite{MiTa16} and  \cite{MiTa17}).

In contrast to the above-mentioned properties of maximal pro-$p$ Galois groups, Frattini-injectivity (and also Frattini-resistance) is a fairly elementary and quite palpable group theoretic condition; 
yet, it seems to be highly restrictive. For instance, within the classes of $p$-adic analytic and solvable pro-$p$ groups, Frattini-injectivity completely characterizes maximal pro-$p$ Galois groups.

\begin{cor}
\label{Galois-analytic}
Let $G$ be a solvable or $p$-adic analytic pro-$p$ group. Then, $G$ is Frattini-injective if and only if it is isomorphic to $G_k(p)$ for some field $k$ that contains a primitive $p$th root of unity, and also contains $\sqrt{-1}$ if $p=2$. 
\end{cor}

In particular, we recover a result due to Ware \cite{Ware92} (when $p$ is odd and $k$ contains a primitive $p^2$th root of unity; for $p=2$ see \cite{JaWa89} and \cite{Wa79}) and 
Quadrelli  \cite{Qu14} (for all $k$; see also \cite{Qu20b}).

\begin{cor}
\label{Galois-solvable}
Let $G$ be a solvable or $p$-adic analytic pro-$p$ group. Then $G$ can be realized as a maximal pro-$p$ Galois group of some field $k$ that contains a primitive $p$th root of unity (and also $\sqrt{-1} \in k$ if $p=2$) if and only if it is free abelian 
or isomorphic to a semidirect product $\langle x \rangle \ltimes A$, where $\langle x \rangle \cong \Z_p$, $A$ is a free abelian pro-$p$ group and 
$x$ acts on $A$ as scalar multiplication by $1+p^s$, with $s \geq 1$ if $p>2$, and $s \geq 2$ if $p=2$.
\end{cor}

As a corollary of Theorem~\ref{max-normal}, we obtain yet another well-known result on maximal pro-$p$ Galois groups due to Engler and Nogueira \cite{EngNog} (for $p=2$) and Engler and Koenigsmann \cite{EngKoen} (for $p>2$).

\begin{cor}
Let $k$ be a field containing a primitive $p$th root of unity (and $\sqrt{-1} \in k$ if $p=2$). Then $G_k(p)$ contains a unique maximal normal abelian subgroup.
\end{cor}

More recently, $1$-smooth cyclotomic pro-$p$ pairs, a formal version of Hilbert~$90$ for pro-$p$ groups (for a precise definition, see Section~\ref{Maximal pro-$p$ Galois groups}), have been investigated in an attempt to
abstract essential features of maximal pro-$p$ Galois groups (see \cite{CleFlo}, \cite{EfQu19}, \cite{QuWe18}, \cite{Qu20a} and \cite{Qu20b}).

\begin{thm}\label{1-smooth}
Let $\mathcal{G}=(G, \theta)$ be a torsion-free $1$-smooth cyclotomic pro-$p$ pair. If $p=2$, in addition, assume that $\textrm{Im}(\theta) \leq 1+4\Z_2$.
Then $G$ is strongly Frattini-resistant. 
\end{thm}

Consequently, many of the known properties of $1$-smooth cyclotomic pro-$p$ pairs can be obtained as consequences of Frattini-injectivity (see Section~\ref{Maximal pro-$p$ Galois groups}).

\medskip

\textbf{Outline of the paper:} In Section~\ref{basic results}, several elementary results on Frattini-injective pro-$p$ groups are established. 
Theorem~\ref{main} is proved in Section~\ref{p-adic analytic}. In Section~\ref{Hierarchical triples}, the concepts of Frattini-resistance and commutator-resistance are developed within the unifying framework of hierarchical triples.
The proofs of Theorem~\ref{solvable} and Theorem~\ref{max-normal} are given in Section~\ref{Frattini-injective solvable}. Section~\ref{Free and Demushkin} is devoted to Free pro-$p$ and Demushkin groups.
In Section~\ref{Maximal pro-$p$ Galois groups}, Frattini-injectivity is investigated in the context of Galois theory. We close the paper with a brief section on another related concept, $p$-power resistance, and a 
section in which we formulate several problems that we hope will stimulate further research on Frattini-injective pro-$p$ groups.    

\medskip

\textbf{Notation:} We take all group theoretic terms in the appropriate sense for topological groups; for instance, subgroups are assumed to be closed, homomorphisms are continuous, and generators are always understood to be topological generators.
Let $G$ be a pro-$p$ group, $H$ a subgroup of $G$ and $x, y \in G$. We use the following fairly standard notation: $d(G)$ is the cardinality of a minimal generating set for $G$;  $x^y=y^{-1}xy$ and
$[x,y]=x^{-1}x^y$; the $n$th terms of the derived series and the lower central series of $G$ are denoted by $G^{(n)}$ and $\gamma_n(G)$, respectively, with the exception of the commutator subgroup, which is always denoted by $[G,G]$; we write $G^{ab}$ for the abelianization of $G$;
the center of $G$ is denoted by $Z(G)$; $N_G(H)$ and $C_G(H)$ are the normalizer and the centralizer of $H$ in $G$, respectively; $G^p$ is the subgroup of $G$ generated by $p$th powers of elements of $G$;
the terms of the lower $p$-series are denoted by $P_i(G)$, so $P_1(G)=G$ and $P_{i+1}(G)=P_i(G)^p[G, P_i(G)]$ for $i \geq 1$.


\section{Basic properties of Frattini-injective pro-$p$ groups}
\label{basic results}

Frattini-injectivity is obviously a hereditary property, that is, every subgroup of a Frattini-injective pro-$p$ group is Frattini-injective. 
Furthermore, a Frattini-injective pro-$p$ group is necessarily torsion-free: if a pro-$p$ group $G$ has a non-trivial element of finite order, then it has an element, say $x$,  of order $p$ and $\Phi(\langle x \rangle)=\{1_G\}=\Phi(\{1_G\})$.
Hence, the only Frattini-injective finite $p$-group is the trivial group, which henceforth will be tacitly disregarded.

\begin{lem}\label{normal}
Let $G$ be a  Frattini-injective pro-$p$ group, and let $H$ be a finitely generated subgroup of $G$.
Then $N_G(H)=N_G(\Phi(H))$. In particular, $H \unlhd G$ if and only if $\Phi(H) \unlhd G$.
\end{lem}
\begin{proof}
Since $\Phi(H)$ is a characteristic subgroup of $H$, it follows that $N_G(H) \leq N_G(\Phi(H))$. For the other inclusion, let $x \in G \setminus N_G(H)$; 
then $H \neq x^{-1}Hx$ and by Frattini-injectivity $ \Phi(H) \neq \Phi(x^{-1}Hx) = x^{-1} \Phi(H)x$. Hence, $x \notin N_G(\Phi(H))$.
\end{proof}


A subgroup $H$ of a pro-$p$ group $G$ is said to be \emph{isolated} (in G) if  $x \in H$ whenever $x^p \in H$. 
(More generally, the condition implies that for every $\alpha \in \mathbb{Z}_p \setminus \{0\}$, if $x^{\alpha} \in H$, then $x \in H$.) 
A related concept  (which will appear only later) is the \emph{isolator} of a subgroup $H$ of a pro-$p$ group $G$;
it is the smallest isolated subgroup of $G$ containing $H$. 

\begin{pro}
\label{isolated}
Every maximal abelian subgroup of a Frattini-injective pro-$p$ group is isolated. 
\end{pro}
\begin{proof}
Let $G$ be a Frattini-injective pro-$p$ group, and let $A$ be a maximal abelian subgroup of $G$.
Consider an element $x \in G$ such that $x^p \in A$, and set $H:=\langle x, A \rangle$. Then $x^p \in Z(H)$, and thus
$\Phi(\langle x \rangle)=\langle x^p \rangle \unlhd H$. It follows from Lemma~\ref{normal} that $\langle x \rangle \unlhd H$. 
This, in turn, implies that $[x, a] \in \langle x  \rangle$ for every $a \in A$.
Consequently,
$$1=[x^p, a]=[x, a]^{x^{p-1}}[x, a]^{x^{p-2}} \cdots [x, a]=[x, a]^p.$$
Since all Frattini-injective pro-$p$ groups are torsion free, it follows that $[x, a]=1$ for every $a \in A$. Hence,  
$H$ is abelian, and as $A$ is a maximal abelian subgroup of $G$, we get that $H=A$, i.e., $x \in A$.
\end{proof}

\begin{cor}
\label{virtually-abelian}
Let $G$ be a Frattini-injective pro-$p$ group, $x, y \in G$, and $\alpha, \beta \in \mathbb{Z}_p \setminus \{0\}$.
The following assertions hold:
\begin{enumerate}[(i)]
\item If $G$ is virtually abelian, then it is abelian. 
\item All centralizers of elements of $G$ are isolated.
\item If $x^{\alpha}$ and $y^{\beta}$ commute, then $x$ and $y$ commute.
\item If $x^{\alpha}=y^{\alpha}$, then $x=y$, i.e., $G$ has the unique extraction of roots property. 
\end{enumerate}
\end{cor}
\begin{proof}
$(i)$ is an immediate consequence of Proposition~\ref{isolated}. In order to prove $(ii)$, 
suppose that $y^p \in C_G(x)$. Then $H:=\langle x, y^p \rangle$ is an abelian group, and it is contained in some
maximal abelian subgroup $A$ of $G$. By Proposition~\ref{isolated}, $y \in A$, and hence $y \in C_G(x)$. 

Suppose that $x^{\alpha}$ and $y^{\beta}$ commute. Then $y^{\beta} \in C_G(x^{\alpha})$, and it follows from $(ii)$ that $y \in C_G(x^{\alpha})$. 
By applying the same argument to the pair of commuting elements $y$ and $x^{\alpha}$, we conclude that $x$ and $y$ commute, which proves $(iii)$.

Now, $(iv)$ is a consequence of $(iii)$ and the fact that
free abelian groups of rank two have the unique extraction of roots property.
\end{proof}

\begin{rem}
Note that all of the claims collected in Corollary~\ref{virtually-abelian} hold true for torsion-free (pro-$p$) 
groups all of whose maximal abelian subgroups are isolated. Furthermore, a pro-$p$ group $G$ has the unique extraction of roots property
if and only if every maximal abelian subgroup of $G$ is isolated. Indeed, suppose that $G$ is a pro-$p$ group with the unique extraction of roots property. Then, $G$ is torsion free and distinct cyclic subgroups of $G$ 
have distinct Frattinis. As in  the proof of Lemma~\ref{normal}, we deduce that for every cyclic subgroup $H$ of $G$, if $\Phi(H) \unlhd G$, then $H \unlhd G$. 
Now observe that besides torsion-freeness, this is the only other property of Frattini-injective pro-$p$ groups used in the proof of Proposition~\ref{isolated}.   
\end{rem}

\begin{lem}
\label{monotone-f.g.}
Let $G$ be a finitely generated pro-$p$ group, and let $M$ be a maximal subgroup of $G$. 
If $d(M) < d(G)$, then $d(G) = d(M) + 1$ and $\Phi(G) = \Phi(M)$. 
\end{lem}
\begin{proof}
Suppose that $d(M)<d(G)$, and let $x \in G \setminus M$. Since $G=\langle x, M \rangle $, we must have $d(G) = d(M) + 1$. 
Furthermore,
$$|G: \Phi(M)| = |G:M ||M:\Phi(M)| = p^{1 + d(M)} = p^{d(G)} =  |G: \Phi(G)|,$$
and as $\Phi(M) \leq \Phi(G)$, it follows that  $\Phi(M) = \Phi(G)$.
\end{proof}

\begin{pro}\label{monotone}
Let $G$ be a Frattini-injective pro-$p$ group. Then, every finitely generated subgroup $H$ of $G$ satisfies the following property: if $K$ is an open subgroup of $H$, then $d(K)\geq d(H) $.
\end{pro}
\begin{proof}
This follows directly from Lemma~\ref{monotone-f.g.} by induction on the index of the open subgroup.
\end{proof}

\section{Frattini-injective $p$-adic analytic pro-$p$ groups}
\label{p-adic analytic}

A pro-$p$ group $G$ is said to be \emph{powerful} if $p$ is odd and $[G,G]\leq G^p$, or $p=2$ and $[G,G]\leq G^4$.   
A finitely generated powerful pro-$p$ group $G$ such that ${\lvert P_i(G) : P_{i+1}(G) \rvert = \lvert G : P_2(G) \rvert}$ for all $i \in \N$ is called \emph{uniform}. By \cite[Theorem~4.5]{DidSMaSe99}, a powerful finitely generated pro-$p$ group is uniform if and only if it is torsion-free. 
Uniform pro-$p$ groups play a central role in the theory of  $p$-adic analytic groups: A topological group is $p$-adic analytic if and only if it contains an open uniform pro-$p$ subgroup 
(see \cite[Theorems 8.1 and 8.18]{DidSMaSe99}). 

In the seminal paper \cite{La65},  Lazard defined \emph{saturable} pro-$p$ groups. For our purposes, it is enough to know that every uniform pro-$p$ group is saturable. To every saturable pro-$p$ group $G$ one can associate a (saturable) $\mathbb{Z}_p$-Lie algebra $L_G$. Moreover, the assignment 
$G \mapsto L_G$ defines an equivalence between the category of saturable pro-$p$ groups and the category of saturable $\mathbb{Z}_p$-Lie algebras. 
One advantage of working with saturable pro-$p$ groups stems from the fact that every torsion-free $p$-adic analytic pro-$p$ group of dimension less than $p$ is saturable (see \cite[Theorem~E]{GoKl09}), but in general not uniform. 

A uniform pro-$p$ group is said to be \emph{hereditarily uniform} if all of its open subgroups are also uniform.

\begin{pro}\label{hereditarily}
Every hereditarily uniform  pro-$p$ group is Frattini-injective.
\end{pro}
\begin{proof}
Let $G$ be a hereditarily uniform  pro-$p$ group. Suppose that there are distinct subgroups $H$ and $K$ of $G$ such that $\Phi(H) = \Phi(K)$. 
Without loss of generality, we may assume that there is some $x\in H \setminus K$. Choose an open subgroup $U$ of $G$ such that $K \leq U$ and $x \notin U$.
Then $x^p \in \Phi(H) = \Phi(K)  \leq \Phi(U)$. Since $U$ is uniform, $\Phi(U)=U^p$ and $x^p=z^p$ for some $z \in U$ (\cite[Lemma 3.4]{DidSMaSe99}). 
By unique extraction of roots in $G$ (\cite[Lemma 4.10]{DidSMaSe99}), $x = z$, which yields a contradiction with $x \notin U$.
\end{proof}

The hereditarily uniform pro-$p$ groups were classified in \cite{KlSn2013}.  It turns out that a uniform pro-$p$ group $G$ is hereditarily uniform if and only if 
it has a constant generating number on open subgroups, that is,  $d(U) = d(G)$ for every open subgroup $U$ of $G$ (cf. \cite{KlSn11} and \cite[Corollary 1.12]{KlSn2013}). 

\begin{pro}\label{virtually}
Let $G$ be a Frattini-injective  $p$-adic analytic  pro-$p$ group. Then, $G$ is virtually hereditarily uniform. More precisely, $G$ contains an open subgroup $U$ isomorphic to one of the following groups:
 \begin{enumerate}
  \item the abelian group $\Z_p^d$ with $d \geq 1$;
  \item the metabelian group $\langle x \rangle \ltimes \Z_p^{d}$, where $d \geq 1$, $\langle x \rangle \cong \Z_p$,
   and $x$ acts on $\Z_p^{d}$ as scalar multiplication by
    $\lambda$, with $\lambda = 1+p^s$ for some $s \geq 1$ if $p>2$,
    and $\lambda = 1 + 2^s$ for some $s \geq 2$ if $p=2$.
 \end{enumerate}
\end{pro}
\begin{proof}
It follows from \cite[Theorem 8.32]{DidSMaSe99} that there is an open subgroup $U$ of $G$ that is uniform.  For every open subgroup $V$ of $U$, we have $d(V) \geq d(U)$ by Proposition~\ref{monotone}, and also
$d(U) \geq d(V)$ by \cite[Theorem 3.8]{DidSMaSe99}. Hence, $U$ is a uniform  pro-$p$ group with constant generating number on open subgroups. By  \cite[Corollary~2.4]{KlSn11}, $U$ is isomorphic to one of the groups listed in the proposition. 
In particular, $U$ is a hereditarily uniform pro-$p$ group (cf. \cite[Corollary 1.12]{KlSn2013}).
\end{proof}

The rest of this section is devoted to eliminating the adverb ``virtually'' from Proposition~\ref{virtually}.
We begin with several lemmas.

\begin{lem}\label{saturable}
Let $p$ be an odd prime, and let $G=\langle x \rangle \ltimes N$, where $\langle x \rangle \cong \Z_p$ and 
$N \cong \Z_p^{d}$ for some $d \geq 1$, be a saturable pro-$p$ group. Suppose that $x^p$ acts on $N$ as scalar multiplication by $1+p^s$ for some $s \geq 1$. Then, $s \geq 2$ and there is a unit $\alpha \in \mathbb{Z}_p^*$ such that  $x^{\alpha}$ acts on $N$ as scalar multiplication by $1+p^{s-1}$.
 \end{lem}
\begin{proof}
Note that the maximal subgroup $H:=\langle x^{p} \rangle \ltimes N$ of $G$ is uniform, and therefore saturable. Consider the $\mathbb{Z}_p$-Lie algebras $L_G$ and $L_H$ associated to $G$ and $H$, respectively. Then,  $L_H$ is a  maximal subalgebra of $L_G$, and we can  choose elements 
$y_1, y_2, \ldots ,y_{d}$ from $N$ such that $\{ \bar{x}, \bar{y}_1, \bar{y}_2, \ldots, \bar{y}_{d}\}$ is a basis for $L_G$ and 
$\{ p\bar{x}, \bar{y}_1, \bar{y_2},\ldots ,\bar{y}_{d}\}$ is a basis for $L_H$. Furthermore, we can assume that for each $1 \leq i \leq d$, we have 
${[\bar{y}_i, p\bar{x}]}_{\textrm{Lie}} = p^s \bar{y}_i$. Hence, we must have $ {[\bar{y}_i, \bar{x}]}_{\textrm{Lie}} = p^{s-1} \bar{y}_i$.
Moreover, $s-1$ can not be $0$, since in that case $L_G$ would not be residually-nilpotent. Therefore, for some suitable unit 
$\alpha \in \mathbb{Z}_p^*$, the element $x^{\alpha}$ acts on $N$  as scalar multiplication by $1+p^{s-1}$.
\end{proof}

\begin{lem}\label{dimension p-1}
Let $G = \langle x \rangle \ltimes N$, where $\langle x \rangle \cong \Z_p$ and $N \cong \Z_p^{d}$  for some $d \geq 1$. 
Suppose that $x^p$ acts on $N$ as scalar multiplication by $1+p^s$ for some $s \geq 1$ if $p>2$,
and as scalar multiplication by $1 + 2^s$ for some $s \geq 2$ if $p=2$. If  $p + 2  \leq \mathrm{dim}(G)$, then $d(G) > 2$.
 \end{lem}
\begin{proof}
Suppose that $p + 2  \leq \textrm{dim}(G)$  and $d(G) = 2$.  Then $G = \langle x, y \rangle$ for some $ y \in N$. 
From ${(x^p)}^{-1}yx^p = y^{1+p^s}$, we get that the set  $T = \{ x^{-i}yx^i ~ | ~  0  \leq i \leq p-1  \}$ generates $N$. Hence,  $d(N) \leq |T| \leq p <   p+1 \leq d=d(N)$, which yields a contradiction. 
\end{proof}

\begin{lem}\label{dimension p}
Let $p$ be an odd prime, and let $G = \langle x \rangle \ltimes N$, where $\langle x \rangle \cong \Z_p$ and $N \cong \Z_p^{p-1}$, be a Frattini-injective pro-$p$ group. Suppose that $x^p$ acts on $N$ as scalar multiplication by $ 1+p^s$ for some $s \geq 1$. 
Then $d(G) > 2$.
\end{lem}
\begin{proof}
Suppose by way of contradiction that $d(G) = 2$. Then $G = \langle x, y \rangle$ for some $ y \in N$. For $1 \leq i \leq p$,  set $y_i := x^{-(i-1)}yx^{i-1} $.
Since ${(x^p)}^{-1}yx^p = y^{1 + p^s}$, the set $\{y_i \mid 1 \leq i \leq p\}$ generates $N$. Moreover,  after (possibly) replacing  $y$ by $x^{-(i-1)}yx^{i-1}$ for a suitable $1 \leq i \leq p$, 
we may assume that $\{y_1, y_2, \ldots, y_{p-1}\}$ is a basis for $N$. 

Let $\alpha_1, \alpha_2, ..., \alpha_{p-1} \in \mathbb{Z}_p$ be such that $ x^{-1}y_{p-1}x = y_1^{\alpha_1}  y_2^{\alpha_2} \cdots    y_{p-1}^{\alpha_{p-1}}.$  Then
\begin{align*}\label{identity p}
 y_1^{1 + p^s} &= y_1^{x^p} = {(y_{p-1})}^{x^2}= {(y_1^x)}^{\alpha_1}  {(y_2^x)}^{\alpha_2}   \cdots  {(y_{p-2}^x)}^{\alpha_{p-2}} {(y_{p-1}^x)}^{\alpha_{p-1}}\\
 &= {y_2}^{\alpha_1}{y_3}^{\alpha_2}   \cdots   {y_{p-1}}^{\alpha_{p-2}} {({y_1}^{\alpha_1}{y_2}^{\alpha_2}  \cdots {y_{p-1}}^{\alpha_{p-1}})}^{\alpha_{p-1}}\\
 &= {y_1}^{\alpha_1\alpha_{p-1}}{y_2}^{\alpha_1 + \alpha_2\alpha_{p-1}}{y_3}^{\alpha_2 + \alpha_3\alpha_{p-1}}  \cdots  {y_{p-2}}^{\alpha_{p-3} + \alpha_{p-2}\alpha_{p-1}} {y_{p-1}}^{\alpha_{p-2} + \alpha_{p-1}^2}.
\end{align*}
By comparing exponents, we get the following relations:

\begin{equation}\label{system}
\left\{ 
\begin{array}{lcr}
1 + p^s = \alpha_1 \alpha_{p-1} \\ 
0 = \alpha_1 + \alpha_2 \alpha_{p-1} \\ 
\vdots \\ 
0= \alpha_i + \alpha_{i+1} \alpha_{p-1}  \\ 
\vdots  \\ 
0= \alpha_{p-3} + \alpha_{p-2} \alpha_{p-1}  \\ 
0 = \alpha_{p-2} + \alpha_{p-1}^2. 
\end{array}
\right. 
\end{equation}
It readily follows that $\alpha_i  = {(-1)}^i\alpha_{p-1}^{p-i}$ for $1 \leq i \leq p-2$ and $-\alpha_{p-1}^p  = 1 + p^s $. 
If $s=1$, the equation ${-(\alpha_{p-1})}^p  = 1 + p^s$ does not have a solution in $\mathbb{Z}_p$. Hence, we may assume that $s \geq 2$, in which case, there is a unique $\omega \in \mathbb{Z}_p$ such that $-{w}^p  = 1 + p^s$. It follows that $\alpha_i  = {(-1)}^i{ \omega}^{p-i}$ for $1 \leq i \leq p-1$. 

Consider the open subgroup $H:= \langle x, y_1^p \rangle=\langle x \rangle\Phi(N)$ of $G$. We have that $ y_1^{-p}y_2^p=[y_1^p, x] \in \Phi(H)$, and it follows by induction that
$$y_{i+1}^{-p}y_{i+2}^p=y_i^{-p}y_{i+1}^p [y_i^{-p}y_{i+1}^p, x] \in \Phi(H)  \textrm{ for every } 1 \leq i < p-2.$$
Since $\langle y_1^{-p}y_2^p, y_2^{-p}y_3^p, \ldots ,y_{p-2}^{-p}y_{p-1}^p, y_1^{p^2}  \rangle$ has index $p$ in $\Phi(N)$, it follows
that 
$$\Phi(H) =  \langle  x^p, y_1^{-p}y_2^p, y_2^{-p}y_3^p, \ldots ,y_{p-2}^{-p}y_{p-1}^p, y_1^{p^2}   \rangle.$$
Next, consider the open subgroup $K := \langle xy_2, y_1^p \rangle$ of $G$. Observe that $K \neq H$.  For every $ \tilde{y} \in N $, we have $[ \tilde{y}, xy_2] = [ \tilde{y}, x]$. Hence, 
as for the subgroup $H$, we can deduce that $y_{i+1}^{-p}y_{i+2}^p  \in  \Phi(K)$ for every $1 \leq i < p-2$. 
Using the identities $y_i^{\beta}x = xy_{i+1}^{\beta}$  ($1 \leq i \leq p-2$) and $y_{p-1}^{\beta} x = x{(y_1^{\alpha_1}  y_2^{\alpha_2} \cdots    y_{p-1}^{\alpha_{p-1}})}^{\beta}$, we obtain
\begin{align*}
{(xy_2)}^p &= x^2y_3xy_3 \cdots xy_3y_2 = x^3y_4xy_4 \cdots xy_4y_3y_2  = \cdots  = x^{p-2}y_{p-1}xy_{p-1}x(y_{p-1} \cdots y_3 y_2)\\ 
&=x^{p-1}y_1^{\alpha_1}  y_2^{\alpha_2} \cdots  y_{p-2}^{\alpha_{p-2}}  (y_{p-1}^{\alpha_{p-1}}x)y_1^{\alpha_1}  y_2^{\alpha_2} \cdots    y_{p-1}^{\alpha_{p-1}}(y_{p-1} \cdots y_3 y_2)  \\ 
&= x^{p-1}y_1^{\alpha_1}  y_2^{\alpha_2} \cdots  y_{p-2}^{\alpha_{p-2}} x {(y_1^{\alpha_1}  y_2^{\alpha_2} \cdots    y_{p-1}^{\alpha_{p-1}})}^{\alpha_{p-1}}y_1^{\alpha_1}  y_2^{\alpha_2} \cdots    y_{p-1}^{\alpha_{p-1}}(y_{p-1} \cdots y_3 y_2)  \\ 
&= x^py_2^{\alpha_1}  y_3^{\alpha_2} \cdots  y_{p-1}^{\alpha_{p-2}}{(y_1^{\alpha_1}  y_2^{\alpha_2} \cdots    y_{p-1}^{\alpha_{p-1}})}^{1+\alpha_{p-1}}(y_{p-1}y_{p-2} \cdots y_3 y_2)   \\
&= x^p{y_1}^{\alpha_1 + \alpha_1\alpha_{p-1}}{y_2}^{1 + \alpha_1  + \alpha_2 + \alpha_2\alpha_{p-1}}{y_3}^{1 + \alpha_2  + \alpha_3 + \alpha_3\alpha_{p-1}} \cdots {y_{p-1}}^{1 + \alpha_{p-2}  + \alpha_{p-1} + \alpha_{p-1}^2}.
\end{align*}
Moreover, it follows from  (\ref{system}) that
$${(xy_2)}^p = x^p{y_1}^{1 + \alpha_1  + p^s}{y_2}^{1 +  \alpha_2} {y_3}^{1 +  \alpha_3}  \cdots   {y_{p-2}}^{1 +  \alpha_{p-2}} {y_{p-1}}^{1 + \alpha_{p-1}}.$$ 
Observe that $\omega \equiv -1 \textrm{ (mod } p\textrm{)}$, and thus $\alpha_i \equiv -1 \textrm{ (mod } p\textrm{)}$ for every $1 \leq i \leq p-1.$
Let $l_i  \in  \Z_p$ be such that $1 +  \alpha_i = pl_i$ ($1 \leq i \leq p-1$). Then
\begin{align*}
{(xy_2)}^p {(y_{p-2}^py_{p-1}^{-p})}^{l_{p-1}} {(y_{p-3}^py_{p-2}^{-p})}^{l_{p-1} + l_{p-2}}  \cdots   {(y_2^py_3^{-p})}^{l_{p-1} + l_{p-2} + \cdots + l_3} {(y_1^py_2^{-p})}^{l_{p-1} + l_{p-2} + \cdots + l_2}&\\ 
= x^py_1^{( \alpha_1 + 1 + p^s) + (1 +  \alpha_2) + (1 +  \alpha_3) +    \cdots      + (1 +  \alpha_{p-2}) + (1 +  \alpha_{p-1})} = x^py_1^{\gamma}  \in  \Phi(K)&,
\end{align*}
where $\gamma=( \alpha_1 + 1 + p^s) + (1 +  \alpha_2) + (1 +  \alpha_3) +    \cdots      + (1 +  \alpha_{p-2}) + (1 +  \alpha_{p-1})$.

For $1 \leq m \leq \frac{p-1}{2}$, we have $1+\alpha_{p-(2m-1)}=1 + {\omega}^{2m-1} =  (1 + \omega) u_m$, where $u_m = 1 - \omega + {\omega}^2 - \cdots -{\omega}^{2m-3} + {\omega}^{2(m-1)}$, and
$1+\alpha_{p-2m} =1 - {\omega}^{2m} = (1 + \omega) v_m$,  where $v_m = (1-\omega)(1 +  {\omega}^2 + {\omega}^4 +  \cdots + {\omega}^{2(m-1)})$.
From  ${(-1)}^i{ \omega}^{p-i}   \equiv -1 ( \textrm{ mod } p)$, we get that
$$u_m \equiv 2m-1 \textrm{ (mod } p\textrm{)} \textrm{ and } v_m  \equiv 2m \textrm{ (mod } p\textrm{)}.$$
Hence, 
\begin{align*}
\tilde{\gamma} &:=  v_{\frac{p-1}{2}} +   u_{\frac{p-1}{2}} +   v_{\frac{p-3}{2}} +   u_{\frac{p-3}{2}} + \cdots + v_1 + u_1 \\
&\equiv (p-1) + (p-2) +  \cdots + 2 + 1 = \frac{(p-1)p}{2} \equiv  0 \textrm{ (mod } p\textrm{)}.
\end{align*}
Since $\gamma = p^s + (1 + \omega)\tilde{\gamma}$ and $s \geq 2$, it follows that $p^2$ divides $\gamma$. Therefore, $x^p \in \Phi(K)$. 
Now it is easy to see that 
$$ \Phi(K) =  \langle  x^p, y_1^{-p}y_2^p, y_2^{-p}y_3^p, ..., y_{p-2}^{-p}y_{p-1}^p, y_1^{p^2}   \rangle  = \Phi(H),$$ 
which yields a contradiction. Hence, we must have $d(G) > 2$.
\end{proof}

\begin{lem}\label{dimension p1}
Let $G = \langle x \rangle \ltimes N$, where $\langle x \rangle \cong \Z_p$ and $N \cong \Z_p^{p}$, be a Frattini-injective pro-$p$ group. 
Suppose that $x^p$ acts on $N$ as scalar multiplication by $ 1+p^s$ for some $s \geq 1$ if $p>2$,
and as scalar multiplication by $1 + 2^s$ for some $s \geq 2$ if $p=2$.  Then $d(G) > 2$.
 \end{lem}
\begin{proof}
Suppose that $d(G) = 2$. Then $G = \langle x, y \rangle$ for some $ y \in N$, and $\{x^{-i}yx^i \mid  0  \leq i \leq p-1 \}$ is a basis for $N$. 
Set $y_1 := y$ and $y_{i+1} := [y_i, x]$ for $ {1  \leq i \leq p-1}$. Then $\{  y_1, y_2, ..., y_{p}  \}$ is also a basis for $N$. 

It is easy to see that $y_2^{x^k} = y_2^{\binom{k}{0}} y_3^{\binom{k}{1}}   \cdots   y_{i+2}^{\binom{k}{i}} \cdots y_{k+2}^{\binom{k}{k}}$ for every 
$1  \leq k \leq p-2$, and 
\begin{align*}
y_2^{x^{p-1}} &=  {(y_2^{\binom{p-2}{0}} y_3^{\binom{p-2}{1}}   \cdots   y_{i+2}^{\binom{p-2}{i}} \cdots y_{p}^{\binom{p-2}{p-2}})}^x \\
&=y_2^{\binom{p-2}{0}}    y_3^{\binom{p-2}{0} + \binom{p-2}{1}}  \cdots   y_{i+2}^{\binom{p-2}{i-1} + \binom{p-2}{i}} \cdots     y_{p-1}^{\binom{p-2}{p-4} + \binom{p-2}{p-3}}  y_{p}^{\binom{p-2}{p-3}} y_{p}^x\\
&=y_2^{\binom{p-2}{0}}    y_3^{\binom{p-1}{1}}  \cdots   y_{i+2}^{\binom{p-1}{i}} \cdots     y_{p-1}^{\binom{p-1}{p-3}}  y_{p}^{\binom{p-2}{p-3}} y_{p}^x.
\end{align*}
Moreover,
\begin{equation}\label{identity p}
 y_1^{p^s} = [y_1, x^{p}] = [y_1, x] {[y_1, x]}^{x} \cdots  {[y_1, x]}^{x^{p-2}} {[y_1, x]}^{x^{p-1}} = y_2 y_2^{x} \cdots   y_2^{x^{p-2}} y_2^{x^{p-1}}.
\end{equation}
By first expressing each $ y_2^{x^k} $ ($1 \leq k \leq p-1$) in (\ref{identity p}) in terms of the basis $y_1, \ldots, y_p$ of $N$, and then applying the 
hockey-stick identity, $\sum_{i=k}^n \binom{i}{k}=\binom{n+1}{k+1}$, to simplify exponents, we obtain
$ y_1^{p^s} = y_2^{p} y_3^{\binom{p}{2}}   \cdots   y_{i}^{\binom{p}{i-1}} \cdots y_{p-1}^{\binom{p}{p-2}}y_p^{p-1}y_{p}^x$. Hence, 
\begin{equation}\label{power}
[y_{p}, x] =  y_1^{p^s} y_2^{-p} y_3^{-\binom{p}{2}}   \cdots   y_{i}^{-\binom{p}{i-1}} \cdots y_{p-1}^{-\binom{p}{p-2}}y_p^{-p}.
\end{equation} 

\emph{Case 1:} $s \geq 2$. Consider the open subgroup $H :=   \langle x, y_1^{p}, y_2 \rangle$ of $G$. It is not difficult to see that $d(H)=3$.   
However, $K:=\langle x, y_1^{p} \rangle$ is an open subgroup of $H$ with $d(K)< d(H)$, which  contradicts Proposition~\ref{monotone}.

\emph{Case 2:} $s=1$ (and thus $p > 2$). Consider the subgroups $H:= \langle x, y_{p} \rangle$ and ${K:= \langle xy_{p-1}, y_p \rangle}$ of $G$. Observe that $H \neq K$.
It is straightforward to see that 
\[\Phi(N)=\langle [y_p,x], [[y_p,x], x], \ldots, [y_p,_p x] \rangle=\langle [y_p,xy_{p-1}], \ldots, [y_p,_p xy_{p-1}] \rangle. \]
Hence, $\Phi(N)$ is a subgroup of both $\Phi(H)$ and $\Phi(K)$. Moreover, it is not difficult to see that $\Phi(H)=\langle x^p, \Phi(N) \rangle$.
Since $\gamma_3(\langle x, y_{p-1} \rangle) \leq \Phi(N)$, it follows from the Hall-Petresco formula that
\[(xy_{p-1})^p \equiv x^p \textrm{ (mod } \Phi(N) \textrm{)}.\]
This implies that $\Phi(K)=\langle x^p, \Phi(N) \rangle =\Phi(H)$, a contradiction.
\end{proof}

\begin{lem}\label{hereditarily 2 dim}
Let $G=\langle x \rangle \ltimes \langle y \rangle$, where $\langle x \rangle \cong \langle y \rangle \cong \Z_2$ and $x$ acts on $\langle y \rangle$ either as scalar multiplication by $ - (1 + 2^s)$ for some $s \geq 2$ or by inversion. Then $G$ is not Frattini-injective.
\end{lem}
\begin{proof}
Suppose first that $x$ acts on $\langle y \rangle$ as scalar multiplication by $ - (1 + 2^s)$ for some $s \geq 2$.
Consider the subgroups $H= \langle x, y^{2^{s-1}} \rangle$ and  $K= \langle xy, y^{2^{s-1}} \rangle$ of $G$. Obviously, $H \neq K$.
Moreover, $\Phi(H)=\langle x^2, y^{2^s} \rangle$ and $\Phi(K)=\langle (xy)^2, y^{2^s} \rangle$. Since 
$(xy)^2=x^2y^{-2^s}$, it follows that $\Phi(H)=\Phi(K)$. Therefore, $G$ is not Frattini-injective.

Now suppose that $x$ acts on $\langle y \rangle$ by inversion. Then $(xy)^2=x^2$, and thus
$\Phi(\langle x \rangle) = \Phi(\langle xy \rangle)$. Hence, in this case also $G$ is not Frattini-injective.
\end{proof}

\medskip

\begin{proof}[Proof of Theorem~\ref{main}]
One implication follows from Proposition~\ref{hereditarily}. For the other implication, suppose that $G$ is Frattini-injective.  By Proposition~\ref{virtually}, $G$ contains an open hereditarily uniform subgroup $U$. 
If $U$ is abelian, then by Corollary~\ref{virtually-abelian} $(i)$, $G \cong \mathbb{Z}_p^d$.  Hence, we may assume that $U =\langle y \rangle \ltimes N$, where  $\langle y \rangle \cong \Z_p$, $N \cong \Z_p^{d-1}$, and $y$ acts on $N$ as scalar multiplication by
$\lambda = 1+p^s$ for some $s \geq 1$ (or $s \geq 2$ if $p=2$).

We proceed by induction on $|G:U |$. If $G=U$, there is nothing to prove; so, suppose that $|G:U |\geq p $. In fact, running along a subnormal series from $U$ to 
$G$, it suffices to consider the case $|G:U| = p$. 

Let $x \in G \setminus U$; then $G = \langle x \rangle U$ and $x^p \in U$. Note that $N$ is the isolator of the commutator subgroup $[U, U]$ of $U$. Hence, $N$ is a characteristic subgroup of $U$. 
Since $U$ is normal in $G$, it follows that $N$ is also normal in $G$. 

Consider the group $K = \langle x \rangle N$. We consider two separate cases: $x^p \in N$ and $x^p \notin N$.

\emph{Case 1:} $x^p \in N$. Then, $K$ is a Frattini-injective pro-$p$ group that contains an abelian subgroup $N$ of index $p$. By Corollary~\ref{virtually-abelian} (i), $K \cong \Z_p^{d-1}$. 
From $\Phi( \langle y^{-1}xy \rangle ) = \langle y^{-1}x^py \rangle =  \langle {(x^p)}^{ \lambda}  \rangle  = \Phi( \langle x^{ \lambda} \rangle )$, it follows that $y^{-1}xy \in  \langle x \rangle$. 
Moreover, we have that ${(x^{ \lambda})}^p = {(x^p)}^{ \lambda} =  y^{-1}x^py =  {(y^{-1}xy)}^p$. By Corollary~\ref{virtually-abelian}~$(iv)$, $y^{-1}xy = x^{ \lambda}$. Therefore,  $y$ acts on $K$ as scalar multiplication by $\lambda$ and $G=\langle y \rangle  \ltimes K \cong U$ is of the required form. 

\emph{Case 2:} $x^p \notin N$.  Then,  $x^p = y^{p^k}w$ for some $k \in \mathbb{N}$ and $w \in N$. 

\emph{Subcase 2.1:} $p$ is odd and $k \geq 1$. Since $N$ is characteristic in $U$, conjugation by $x$ induces an action on $U/N \cong \mathbb{Z}_p$. 
Moreover, as $\textrm{Aut}(\mathbb{Z}_p) \cong C_{p-1} \times \mathbb{Z}_p$, this action must be trivial. Put $z := x^{-1}y^{p^{k-1}}$; then 
$z^p = {(x^{-1}y^{p^{k-1}})}^p \equiv x^{-p}y^{p^{k}}  \equiv 1  \textrm{ (mod } N \textrm{)}$. 
Hence,  $z^p \in N$, and after replacing $x$ by $z$, we return to Case 1.  

\emph{Subcase 2.2:} $p=2$ and $k \geq 1$. Since $\textrm{Aut}(\mathbb{Z}_2) \cong C_{2} \times \mathbb{Z}_2$, either 
$y^x \equiv y \textrm{ (mod } N \textrm{)}$ or $y^x  \equiv y^{-1} \textrm{ (mod } N \textrm{)}$. We contend that the latter case does not occur.  Indeed, suppose that  $y^x  \equiv y^{-1} \textrm{ (mod } N \textrm{)}$; thus $y^x  = y^{-1}n_0$ for some $n_0 \in N$. Then, for every $n\in N$, we have
$$  {(n^x)}^{1 +  2^s} =  {(n^{1 +  2^s})}^x = (n^{y})^{x} = (n^{x})^ {y^{x}} =  (n^{x})^ {y^{-1}n_0} = (n^{x})^ {y^{-1}} =  (n^{x})^{ {(1 +  2^s)}^{-1}}.$$  Hence, ${(1 +  2^s)}^2 = 1$, a contradiction.

Thus we must have $[y, x] \in N$. Put $z:= x^{-1}y^{2^{k-1}}$; then $z^2 = {(x^{-1}y^{2^{k-1}})}^2 \equiv x^{-2}y^{2^{k}}  \equiv 1  
\textrm{ (mod } N \textrm{)}$. Therefore,  $z^2 \in N$, and after replacing $x$ by $z$, we return to Case 1. 

\emph{Subcase 2.3:} $k = 0$. Replacing $y$ by $yw$, we may assume that  $x^p = y$. We proceed by induction on $d=\textrm{dim}(G)$.
Since $N$ is normal in $G$, we conclude that  $G = \langle x \rangle \ltimes N$.
If $p > d$, then $p$ is odd  and by \cite[Theorem E]{GoKl09}, $G$ is saturable. It follows from Lemma~\ref{saturable} that $G$ is of the required form. 

If $p=d=2$, then by the classification of $2$-adic analytic pro-$2$ groups (\cite[Proposition~7.2]{GoKl09}), either $G$ is of the required form or 
$G \cong \Z_2 \ltimes \Z_2 $, where the action is as scalar multiplication by $-(1+2^t)$ for some $t \geq 2$ or by inversion. 
The latter case is excluded by Lemma~\ref{hereditarily 2 dim}.

Therefore, we may assume that $d \geq \max \{p, 3\}$. Let   $\{  z_1, z_2, ...,z_{d-1} \}$ be a basis for $N$.  
For each $1 \leq i \leq d-1$, set $N_i := \langle x, z_i  \rangle$. We claim that each $N_i$ is of infinite index in $G$. 
Indeed, if $N_i$ is open in $G$, then $\textrm{dim}(N_i)=\textrm{dim}(G)=d$, and it follows from Lemma~\ref{dimension p-1},  Lemma~\ref{dimension p} and Lemma~~\ref{dimension p1} that $d(N_i) >2$, a contradiction. 
   
Hence, each $N_i$ is Frattini-injective and of dimension $\leq d-1$. By the induction hypothesis,  $N_i$ is of the required form 
(hereditarily uniform). It follows easily (by Lie theoretic methods, for example) that $s \geq 2$ and for each $1 \leq i \leq d-1$, there is 
$\alpha_i \in \mathbb{Z}_p^{\times}$ such that $x^{-\alpha_i}z_i{x^{\alpha_i}}=z_i^{1+p^{s-1}}$. Since $x^{-p}z_ix^{p}=z_i^{1+p^s}$ for each $1 \leq i \leq d-1$, it is not difficult to see that we must have $\alpha_1=\alpha_2=\ldots=\alpha_{d-1}$. Therefore, $G$ is of the required form.
\end{proof}

We end this section with an example of a $p$-adic analytic pro-$p$ group in which Frattini-injectivity fails in a rather extreme way.
Let $\mathbb{D}_p$ be a central simple $\mathbb{Q}_p$-division algebra of index $2$, ${\Delta}_p$ the (unique) maximal $\mathbb{Z}_p$-order in $\mathbb{D}_p$ and $\mathfrak{P}$ the maximal ideal of ${\Delta}_p$. 
Let $SL_1(\mathbb{D}_p)$ be the set of elements of reduced norm $1$ in $\mathbb{D}_p$, and let $G= SL_1^1({\Delta}_p): = SL_1(\mathbb{D}_p) \cap (1+ \mathfrak{P})$. 
Then, if $p >3$, for every maximal subgroup $M$ of $G$ we have $\Phi(M) = \Phi(G)$   (cf. \cite[Lemma 2.26]{NoSn19}).

\section{Hierarchical triples}
\label{Hierarchical triples}
 
\begin{definition}
Let $G, K$ and $H$ be pro-$p$ groups with $H \leq K \leq G$. We say that $(G, K, H)$ is a \emph{hierarchical triple} if for every $x \in G$, we have that 
$x \in K$ whenever $x^p \in H$.  We call a pro-$p$ group $G$ \emph{Frattini-resistant} if for every finitely generated subgroup $H$ of $G$, the triple $(G,H,\Phi(H))$ is hierarchical.
\end{definition}
 
Recall that a pro-$p$ group $G$ is said to be \emph{strongly Frattini-injective} if distinct subgroups of $G$ (not necessarily finitely generated) have distinct Frattinis. 
In the same vein, $G$ is defined to be \emph{strongly Frattini-resistant} if $(G,H,\Phi(H))$ is a hierarchical triple for every subgroup $H$ of $G$.

Our first result shows that the definition of Frattini-resistance in terms of hierarchical triples coincides with the definition given in the introduction.

\begin{pro}
\label{embedding}
A pro-$p$ group $G$ is Frattini-resistant if and only if for all finitely generated subgroups $H$ and $K$ of $G$,
\[\Phi(H) \leq \Phi(K) \implies H \leq K. \]
In other words, a pro-$p$ group $G$ is Frattini-resistant if and only if the function $H \mapsto \Phi(H)$ is an embedding of the partially ordered set of finitely generated subgroups of $G$ into itself.
\end{pro}
\begin{proof}
Suppose that $G$ is a Frattini-resistant pro-$p$ group, and let $H$ and $K$ be finitely generated subgroups of $G$ such that $\Phi(H) \leq \Phi(K)$.  
For every $x \in H$, we have that $x^p \in \Phi(K)$, and as $(G, K, \Phi(K))$ is a hierarchical triple, it follows that $x \in K$. Hence, $H \leq K$.

For the converse, suppose that $\Phi(H) \leq \Phi(K) \implies H \leq K$ for all finitely generated subgroups $H$ and $K$ of $G$.
Let $L$ be a finitely generated subgroup of $G$. If $x^p \in \Phi(L)$ for some $x \in G$,
then $\Phi(\langle x \rangle)=\langle x^p \rangle \leq \Phi(L)$, and thus $\langle x \rangle \leq L$, i.e., $x \in L$.
It follows that $(G, L, \Phi(L))$ is a hierarchical triple, and therefore $G$ is Frattini-resistant.
\end{proof}

Clearly, there is also a ``strong'' version of Propositio~\ref{embedding}:  A pro-$p$ group $G$ is 
strongly Frattini-resistant if and only if for all subgroups $H$ and $K$ of $G$,
${\Phi(H) \leq \Phi(K) \implies H \leq K}$. 

\begin{cor}
\label{FR}
Every (strongly) Frattini-resistant pro-$p$ group is (strongly) Frattini-injective. 
\end{cor}  
\begin{proof}
Let $G$ be a Frattini-resistant pro-$p$ group, and suppose that $\Phi(H)=\Phi(K)$ for some finitely generated subgroups $H$ and $K$ of $G$.
By Proposition~\ref{embedding}, $\Phi(H) \leq \Phi(K)$ implies $H \leq K$, and $\Phi(K) \leq \Phi(H)$ implies $K \leq H$. Therefore, $H=K$.

In like manner, the ``strong'' version of the corollary is a consequence of the ``strong'' version of Proposition~\ref{embedding}.  
\end{proof}

We develop next several results that could be useful when trying to prove that a given pro-$p$ group is (strongly) Frattini-resistant.

\begin{pro}
\label{open-resistant}
Let $G$ be a pro-$p$ group, and suppose that $(G, U, \Phi(U))$ is a hierarchical triple for every open subgroup $U$ of $G$. Then,
$G$ is strongly Frattini-resistant.
\end{pro}
\begin{proof}
Let $H$ be a proper subgroup of $G$, and let $x \in G \setminus H$. 
Then, there exists an open subgroup $U$ of $G$ such that $H \leq U$ and $x \notin U$. By assumption, $(G, U, \Phi(U))$ is a hierarchical triple; 
hence, $x^p \notin \Phi(U)$. Since $\Phi(H) \leq \Phi(U)$, it follows that $x^p \notin \Phi(H)$. Therefore, $(G, H, \Phi(H))$ is a hierarchical triple.  
\end{proof}

\begin{cor}
\label{f.g Frattini-resistant}
A finitely generated Frattini-resistant pro-$p$ group is strongly Frattini-resistant.  
\end{cor}
\begin{proof}
This follows from Proposition~\ref{open-resistant} and the fact that every open subgroup of a finitely generated pro-$p$ group is also finitely generated.
\end{proof}

Recall that an epimorphism $\phi:G  \to H$ of pro-$p$ groups such that $\textrm{ker }\phi \leq \Phi(G)$ is called a \emph{Frattini-cover}.

\begin{pro}
\label{F-R criterion}
Let $G$ be a pro-$p$ group. Suppose that for every open subgroup $U$ of $G$, there exists a Frattini-cover $\phi: U \to K$ onto a pro-$p$ group $K$ with the property that
$(K, M, \Phi(M))$ is a hierarchical triple for every maximal subgroup $M$ of $K$. Then, $G$ is strongly Frattini-resistant.
\end{pro}
\begin{proof}
Let $U$ be a proper open subgroup of $G$, and let $x \in G \setminus U$. 
Fix a Frattini-cover $\phi:\langle x, U \rangle \to K$ onto a pro-$p$ group $K$ with the property that
$(K, M, \Phi(M))$ is a hierarchical triple for every maximal subgroup $M$ of $K$. 

Set $N:= \Phi(\langle x, U \rangle)U$; then $N$ is a maximal subgroup of $\langle x, U \rangle$ which contains $U$, but does not contain $x$. 
Since $\phi$ is a Frattini-cover, $M:=\phi(N)$ is a maximal subgroup of $K$ and $\phi(x) \notin M$.
Furthermore, $\phi(x^p)=\phi(x)^p \notin \Phi(M)$ (because $(K, M, \Phi(M))$ is a hierarchical triple), and as $\phi(\Phi(N))=\Phi(M)$, we get that $x^p \notin \Phi(N)$.
Since $\Phi(U) \leq \Phi(N)$, it follows that $x^p \notin \Phi(U)$. Hence, $(G, U, \Phi(U))$ is a hierarchical triple.
As $U$ was chosen to be an arbitrary (proper) open subgroup of $G$, it follows from Proposition~\ref{open-resistant} that  $G$ is strongly Frattini-resistant.
\end{proof}

\begin{rem}
For a pro-$p$ group $G$ that is not finitely generated, the existence of appropriate Frattini-covers (as in Proposition~\ref{F-R criterion}) for all finitely generated subgroups of $G$, 
implies that $G$ is Frattini-resistant (although, not necessarily strongly Frattini-resistant).
\end{rem}

\begin{cor}
\label{open-maximal}
Let $G$ be a pro-$p$ group. Suppose that for every open subgroup $U$ of $G$ and for every maximal subgroup $M$ of $U$, the triple $(U, M, \Phi(M))$ is hierarchical. Then, $G$ is strongly Frattini-resistant.
\end{cor}
\begin{proof}
For every open subgroup $U$ of $G$, the identity map $id_U:U \to U$ is a Frattini-cover satisfying the condition of Proposition~\ref{F-R criterion}.  
\end{proof}

\begin{definition}
We define a pro-$p$ group $G$ to be \emph{strongly commutator-resistant} (\emph{commutator-resistant}) if $(H, \Phi(H), [H,H])$ is a hierarchical triple
for every (finitely generated) subgroup $H$ of $G$. 
\end{definition}

\begin{pro}
Every (strongly) commutator-resistant pro-$p$ group is (strongly) Frattini-resistant.
\end{pro}  
\begin{proof}
Let $G$ be a commutator-resistant pro-$p$ group, and let $H$ be a finitely generated subgroup of $G$. 
Consider an element $x \in G$ such that $x^p \in \Phi(H)=H^p[H,H]$. Then $x^p[H,H]=h^p[H,H]$
for some $h \in H$. Set $K:=\langle x, H \rangle$; as $x, h \in K$ and $x^p[K,K]=h^p[K,K]$, it follows that
$(x^{-1}h)^p \in [K,K]$. Since $K$ is finitely generated, $(K, \Phi(K), [K,K])$ is a hierarchical triple, and thus
$x^{-1}h \in \Phi(K)$, or equivalently, $x\Phi(K)=h\Phi(K)$. Hence, we may replace $x$ by $h$ in a generating set for $K$. It follows that
$K=H$, and thus $x \in H$. This proves that $G$ is Frattini-resistant.

The ``strong'' version of the proposition can be proved in the same way.
\end{proof}

For an element $x$ of a pro-$p$ group $G$, we have that $x \in \Phi(G)$ if and only if $x[G,G] \in \Phi(G^{ab})$ (in other words, the abelianization homomorphism is a Frattini-cover).  
It follows that $(G, \Phi(G), [G,G])$ is a hierarchical triple if and only if every element of $G^{ab}$ of order $p$ is contained in $\Phi(G^{ab})$.

\begin{pro}
\label{commutator-resistant abelianization}
Let $G$ be a pro-$p$ group. The following statements are equivalent:
\begin{enumerate}[(i)]
\item $(G, \Phi(G), [G,G])$ is a hierarchical triple.
\item Every element of order $p$ in $G^{ab}$ is contained in $\Phi(G^{ab})$.
\item Every element of order $p$ in $G^{ab}$ is a $p$th power.
\end{enumerate}
\end{pro}
\begin{proof}
This follows from the remarks made before the proposition and the fact that $\Phi(G^{ab})$ consists of the $p$th powers of elements of $G^{ab}$.
\end{proof}

As an immediate consequence of Proposition~\ref{commutator-resistant abelianization}, we get the following 

\begin{cor}
\label{CR - criteria}
Let $G$ be a pro-$p$ group. The following assertions hold:
\begin{enumerate}[(i)]
\item If $G$ is finitely generated, then $(G, \Phi(G), [G,G])$ is a hierarchical triple if and only if $G^{ab}$ does not contain a direct cyclic factor of order $p$.
\item If $G^{ab}$ is torsion-free, then $(G, \Phi(G), [G,G])$ is a hierarchical triple.
\end{enumerate}
\end{cor}


The following characterization of commutator-resistance is handy within the context of Galois theory.

\begin{pro}
\label{extension}
Let $G$ be a pro-$p$ group, and let $\pi: \Z_p/p^2\Z_p \to \Z_p/p\Z_p$ be the natural projection.
Then, $(G, \Phi(G), [G,G])$ is a hierarchical triple if and only if for every homomorphism $\phi:G \to  \Z_p/p\Z_p$, there is a 
homomorphism $\psi:G \to  \Z_p/p^2\Z_p$ such that $\pi \circ \psi = \phi$.
\end{pro}
\begin{proof}
Suppose that $(G, \Phi(G), [G,G])$ is a hierarchical triple, and let ${\phi:G \to \Z_p/p\Z_p}$ be an epimorphism.
Then $\phi$ factors through an epimorphism ${\bar{\phi}:G^{ab} \to \Z_p/p\Z_p}$.  
Let $x \in G^{ab} \setminus \textrm{ker }\bar{\phi}$. If $px \in \Phi(\textrm{ker } \bar{\phi})$, then $px=py$  for some $y \in \textrm{ker } \bar{\phi}$, and
$p(x-y)=0$; this contradicts Proposition~\ref{commutator-resistant abelianization} $(ii)$ since $x-y \notin \Phi(G^{ab})$.
Hence, there exists a maximal subgroup $M$ of $\textrm{ker } \bar{\phi}$ that does not contain $px$. It follows that $G^{ab}/M \cong \Z_p/p^2Z_p$, and 
there is an obvious homomorphism $\psi:G \to \Z_p/p^2\Z_p$ such that $\pi \circ \psi=\phi$.

For the other direction, suppose that for every homomorphism $\phi:G \twoheadrightarrow  \Z_p/p\Z_p$, there exists a 
homomorphism $\psi:G \to  \Z_p/p^2\Z_p$ such that $\pi \circ \psi = \phi$, and let $x \in {G \setminus \Phi(G)}$. Choose a maximal subgroup 
$M$ of $G$ that does not contain $x$, and consider the quotient homomorphism $\phi: G \to G/M\cong  \Z_p/p\Z_p$ ($\phi(x)=1+p\Z_p$).
Let $\psi:G \to  \Z_p/p^2\Z_p$ be a homomorphism such that $\pi \circ \psi = \phi$. Then $\psi(x)$ generates $\Z_p/p^2\Z_p$, and thus 
$\psi(x^p)=\psi(x)^p \neq 1$. Since $[G,G] \leq \textrm{ker } \psi$, it follows that $x^p \notin [G, G]$. Therefore, $(G, \Phi(G), [G,G])$ is a hierarchical triple.  
\end{proof}

Considering $\Z_p/p^2\Z_p$ as a trivial $G$-module, the extension property of Proposition~\ref{extension} comes down to saying that the 
natural projection $\Z_p/p^2\Z_p \to \Z_p/p\Z_p$ induces an epimorphism  $H^1(G, \Z_p/p^2\Z_p) \to H^1(G, \Z_p/p\Z_p)$ of cohomology groups.

\begin{cor}
\label{commutator open}
Let $G$ be a pro-$p$ group, and suppose that $(U, \Phi(U), [U, U])$ is a hierarchical triple for every open subgroup $U$ of $G$.
Then, $G$ is strongly commutator-resistant. In particular, a finitely generated commutator-resistant pro-$p$ group is strongly commutator-resistant.  
\end{cor}
\begin{proof}
By Proposition~\ref{extension}, the natural projection $\pi: \Z_p/p^2\Z_p \to \Z_p/p\Z_p$ induces an epimorphism  
${\pi^*: H^1(U, \Z_p/p^2\Z_p) \to H^1(U, \Z_p/p\Z_p)}$ for every open subgroup $U$ of $G$ (where $U$ is assumed to act trivially on $\Z_p/p^2\Z_p$). 
It follows that ${\pi^*: H^1(H, \Z_p/p^2\Z_p) \to H^1(H, \Z_p/p\Z_p)}$ is an epimorphism for every subgroup $H$ of $G$ (cf. \cite[I.2.2, Proposition~8]{Se97}).
Hence, by Proposition~\ref{extension}, $G$ is strongly commutator resistant. 
\end{proof}

Before we turn to concrete classes of groups, we make one more useful observation. 

\begin{pro}
\label{inverse limits}
The properties Frattini-injective, Frattini-resistant and commutator-resistant (as well as their ``strong'' forms) are preserved under inverese limits.  
\end{pro}
\begin{proof}
Let $(G_i, \phi_{i,j})_I$ be an inverse system of Frattini-injective pro-$p$ groups with inverse limit $(G, \phi_i)_{i \in I}$.
Let $H$ and $K$ be finitely generated subgroups of $G$ with $\Phi(H)=\Phi(K)$. Then, for every $i \in I$, 
\[\Phi(\phi_i(H))=\phi_i(\Phi(H))=\phi_i(\Phi(K))=\Phi(\phi_i(K)).\]
Hence, $\phi_i(H)=\phi_i(K)$ for all $i \in I$, and thus $H=K$.

It is equally easy to prove that all of the other properties are preserved under inverse limits. 
\end{proof}

\section{Frattini-injective solvable pro-$p$ groups}
\label{Frattini-injective solvable}

Let $G=\langle x \rangle \ltimes \Z_p^{d}$, where $\langle x \rangle \cong \Z_p$ and $x$ acts on $\Z_p^{d}$ as scalar multiplication by $1+p^s$ with $s \geq 1$ ($s \geq 2$ if $p=2$).  
It is easy to see that $G^{ab}=\Z_p \times (\Z_p/p^s\Z_p)^d$. It follows from Corollary~\ref{CR - criteria} $(i)$ that $G$ is not commutator-resistant if $s=1$.
On the other hand, it is not difficult to show that $G$ is commutator-resistant for $s \geq 2$. Moreover, a slight modification of the proof of Proposition~\ref{hereditarily} yields the following

\begin{pro}
\label{p-adic analytic Frattini-resistant}
Every $p$-adic analytic Frattini-injective pro-$p$ group is strongly Frattini-resistant. 
\end{pro}

It follows from Theorem~\ref{main} that all Frattini-injective $p$-adic analytic pro-$p$ groups are metabelian. 
Conversely, we prove in this section that every solvable Frattini-injective pro-$p$ group is metabelian and locally $p$-adic analytic.

\begin{lem}
\label{semidirect product finite}
Let $G= \langle x \rangle \ltimes A$, where $\langle x \rangle \cong \Z_p$ and $A$ is an abelian pro-$p$ group, be a finitely generated Frattini-injective pro-$p$ group. 
Then $A$ is finitely generated.
\end{lem}
\begin{proof}
Since $G$ is finitely generated, $A$ is finitely generated as a topological $\langle x \rangle$-module.
The completed group algebra $\Z_p[\![\langle x \rangle]\!]$ can be identified with the formal power series algebra $\Z_p[\![y]\!]$ (by identifying $x$ with $1+y$);
so, we may regard $A$ as a right (topological) $\Z_p[\![y]\!]$-module (cf. \cite[Chapter 7]{Wil}). 

First suppose that $A$ is a cyclic $\Z_p[\![y]\!]$-module. 
Thus $A \cong \Z_p[\![y]\!]/I$ for some ideal $I$ of $\Z_p[\![y]\!]$.
We claim that $I$ can not be the zero ideal. Indeed, identify $A$ with $\Z_p[\![y]\!]$, and consider the subgroups $H:=\langle x, p^2y^0, py, y^2  \rangle$ and $K:=\langle x, p^2y^0, y^2 \rangle$ of $G=  \langle x \rangle \ltimes \Z_p[\![y]\!]$
(where, in order to avoid confusion, we denote by $y^0$ the identity element of $\Z_p[\![y]\!]$). It is readily seen that $H \neq K$ (in fact, $K$ is a maximal subgroup of $H$), however,
\[\Phi(H)=\Phi(K)=\langle x^p \rangle[p^2\Z_p[\![y]\!] + (y^2)], \]
which contradicts the assumption that $G$ is Frattini-injective.

Hence, we may assume that $I \neq (0)$. Since $A$ is Frattini-injective, and thus torsion-free, there is an element $a(y)=\sum_{n \geq 0} a_ny^n \in I$ that is not $p$-divisible in the abelian group $\Z_p[\![y]\!]$.
Let $m \geq 0$ be the smallest integer such that $a_m \notin p\Z_p$; then $a(y) \equiv b(y) \mod \Phi(\Z_p[\![y]\!])$, where $b(y)=\sum_{n \geq m} a_ny^n$. Since $a_m$ is a unit in $\Z_p$,  
there is $c(y) \in \Z_p[\![y]\!]$ such that $b(y)c(y)=y^m$. Consequently, $a(y)c(y) \equiv y^m \mod \Phi(\Z_p[\![y]\!])$ and $(y^m)+ \Phi(\Z_p[\![y]\!]) \leq I+ \Phi(\Z_p[\![y]\!])$.
From here it readily follows that $A$ is finitely generated as a pro-$p$ group.

\medskip

Now suppose that $A$ is generated by $a_1, a_2, \ldots, a_d$ as $\Z_p[\![y]\!]$-module.
For each $i=1, \ldots, d$, let $B_i$ be the submodule of $A$ generated by $a_i$; put $H_i:=\langle x, B_i \rangle=\langle x \rangle B_i$ and $K_i:=\langle x, B_i^p \rangle=\langle x \rangle B_i^p$. 
It follows from what has been already proved that $B_i$ is a finitely generated free abelian pro-$p$ group. 
Consequently, $H_i$ is a $p$-adic analytic Frattini-injective pro-$p$ group (since every extension of $p$-adic analytic pro-$p$ groups is $p$-adic analytic). 
By Proposition~\ref{p-adic analytic Frattini-resistant}, $H_i$ is strongly Frattini-resistant.

For a given $a \in A$, we have $[x, a] \in A$ and $[[x,a],a]=1$. Hence, for every $n \in \mathbb{N}$, 
\[[x,a^n]=[x,a][x,a^{n-1}][[x,a],a^{n-1}]=[x,a][x,a^{n-1}].\]
It follows that $[x,a^n]=[x,a]^n$ for all $n \in \mathbb{N}$. In particular, for each $i=1, \ldots d$,
\[[x,a_i]^p=[x, a_i^p] \in [K_i,K_i] \leq \Phi(K_i) .\]
Hence, $[x, a_i] \in K_i$ (because $H_i$ is strongly Frattini-resistant). Since also ${[x, a_i] \in B_i}$, we get that
\[[x, a_i] \in K_i \cap B_i=B_i^p \leq A^p.\]
Note that $A$ is generated (as a pro-$p$ group) by $x^{-\alpha}a_ix^{\alpha}$ ($\alpha \in \Z_p$ and $i=1, \ldots, d$).
However, $[x, a_i] \in A^p=\Phi(A)$ implies that $a_i\Phi(A)=x^{-1}a_ix\Phi(A)$. 
Therefore, $a_1, \ldots, a_d$ suffice to generate $A$.
\end{proof}

\begin{lem}
\label{metabelian}
Let $G$ be a non-abelian Frattini-injective metabelian pro-$p$ group. 
Then $G\cong \langle x \rangle \ltimes A$, where $\langle x \rangle \cong \Z_p$, $A$ is a free abelian pro-$p$ group, and 
$x$ acts on $A$ as scalar multiplication by $1+p^s$ with $s \geq 1$ if $p$ is odd, and $s \geq 2$ if $p=2$.   
\end{lem}
\begin{proof}
Let $A$ be a maximal abelian subgroup of $G$ containing $[G,G]$. Then $A$ is a free abelian pro-$p$ group (all torsion-free abelian pro-$p$ groups are free abelian).
Moreover, $A \unlhd G$ (since $[G,G] \leq A$) and $A$ is isolated in $G$ (Proposition~\ref{isolated}).    

Let $x \in G \setminus A$, and let $a_1, \ldots, a_d \in A$. Consider the subgroup $H:=\langle x, a_1, \ldots, a_d \rangle$ of $G$.
Let $N$ be the normal subgroup of $H$ generated (as a normal subgroup) by the elements $a_1, \ldots, a_d$.
Then $N \leq A$, and hence $N$ is abelian. Since $A$ is isolated in $G$, we also have $\langle x \rangle \cap N = \{1\}$.
Therefore, $H=\langle x \rangle N$ is an internal semidirect product. By Lemma~\ref{semidirect product finite},
$N$ is a finitely generated free abelian pro-$p$ group, and thus $H$ is $p$-adic analytic.

By Theorem~\ref{main}, either $H$ is abelian or for some unit $\alpha$ of $\Z_p$, $x^{\alpha}$ acts on $N$ as scalar 
multiplication by $1+p^s$ with $s \geq 1$ if $p$ is odd, and $s \geq 2$ if $p=2$. Since $a_1, \ldots, a_d$ were chosen to be arbitrary elements of $A$, it follows that $x$ must act in the same way on all elements of $A$.
As $A$ is a maximal abelian subgroup of $G$, $x$ can not commute with all elements of $A$; so, $x^{\alpha}$ (for some unit $\alpha$) acts on $A$ as scalar multiplication by $1+p^s$. 

The group $G/A$ is torsion-free since $A$ is isolated in $G$.
We claim that ${G/A \cong \Z_p}$. Suppose that this is not the case. 
Then, there exist $x_1, x_2 \in G$ such that $x_1A$ and $x_2A$ generate in $G/A$ a free abelian pro-$p$ group of rank two.
Fix $a \in A$, $a \neq 1$, and consider the group $L:=\langle x_1, x_2, a \rangle$. 
Now, we know that $x_1$ and $x_2$ normalize the abelian group ${M:=\langle [x_1, x_2], a \rangle \leq A}$. 
Hence, $M \unlhd L$ and $L/M \cong \Z_p \times \Z_p$. This implies that $L$ is $p$-adic analytic. 
It follows from Theorem~\ref{main}, that all Frattini-injective 
$p$-adic analytic pro-$p$ groups that have a quotient isomorphic to $\Z_p \times \Z_p$ are abelian. 
However, $L$ is not abelian since $x_1$ and $x_2$ do not commute with $a$, a contradiction.

Let $x$ be an element of $G$ such that $G/A=\langle xA \rangle.$ 
Clearly, we may choose $x$ in such a way that it acts on $A$ as scalar multiplication by $1+p^s$ ($s \geq 1$ if $p$ is odd, and $s \geq 2$ if $p=2$).
Therefore, $G= \langle x \rangle \ltimes A$ is of the required form. 
\end{proof}

\medskip

\begin{proof}[Proof of Theorem~\ref{solvable}]
Let $G=\langle x \rangle \ltimes A$ be a semidirect product as in the statement of the theorem. It is easily seen that 
$$G=\varprojlim_{i \in I} \langle x \rangle \ltimes A_i,$$ 
where
$\{A_i \mid i \in I\}$ is the set of finitely generated direct factors of $A$. Since all of the groups $\langle x \rangle \ltimes A_i$ are Frattini-injective (Theorem~\ref{main}), it follows from
Proposition~\ref{inverse limits} that $G$ is also Frattini-injective. (In fact, Proposition~\ref{p-adic analytic Frattini-resistant} and Proposition~\ref{inverse limits} imply that $G$ is strongly Frattini-resistant.)

In light of Lemma~\ref{metabelian}, in order to prove the converse, it suffices to argue that there are no solvable Frattini-injective pro-$p$ groups
of derived length $3$. Suppose to the contrary that $G$ is such a group. By Lemma~\ref{metabelian},
$[G,G]=\langle x \rangle \ltimes A$, where $\langle x \rangle \cong \Z_p$, $A$ is a free abelian pro-$p$ group, and 
$x$ acts on $A$ as scalar multiplication by $1+p^s$ with $s \geq 1$ if $p$ is odd, and $s \geq 2$ if $p=2$. 
Note that $A$ is the isolator of $G^{(2)}=A^{p^s}$ in $[G,G]$. Since $[G, G]$ is normal in $G$ and $G^{(2)}$ is characteristic in $[G,G]$, it follows that $A$ is normal in $G$. 

There are elements $y, z \in G$ such that $[y, z] \notin A$ (otherwise, we would have $[G,G] \leq A$ and $G^{(2)}=1$, a contradiction);
so, $[y, z]$ acts on $A$ as scalar multiplication by some $\lambda \neq 1$. Clearly, $\langle y, A \rangle$ and $\langle z, A \rangle$ are metabelian groups, and it follows from
Lemma~\ref{metabelian} that $y$ and $z$ also act on $A$ by scalar multiplication. 

Fix $a \in A$, $a \neq 1$. Then the group $\langle y, z \rangle$ acts on $\langle a \rangle \cong \Z_p$, so we get a homomorphism 
$\phi:\langle y, z \rangle \to \textrm{Aut}(\Z_p)$. However, $[y, z] \notin \textrm{ker} \phi$, which is impossible since $\textrm{Aut}(\Z_p)$ is abelian.
\end{proof}

\medskip
 
\begin{proof}[Proof of Theorem~\ref{max-normal}]
It is readily seen (using Zorn's lemma) that $G$ contains a maximal normal abelian subgroup $N$.
The isolator of $N$ is also a normal subgroup of $G$, and it follows from Lemma~\ref{virtually-abelian} $(iii)$ that it is abelain.
Hence, $N$ coincides with its isolator, and so it is isolated in $G$.

Suppose that $M$ is another maximal normal abelian subgroup of $G$. Then, $NM$ is a solvable Frattini-injective pro-$p$ group. It follows from Theorem~\ref{solvable}
that $NM=\langle x \rangle \ltimes A$ (where the semidirect product is of the form described in the theorem). 
Without loss of generality, we may assume that the restriction to $N$ of the projection homomorphism from $NM$ onto $\langle x \rangle$ is surjective.
Hence, there is $n \in N$ with $n=xy$ for some $y \in A$. 
For every $a \in A$, $[a, n]=a^{p^s} \in N$; as $N$ is isolated in $G$, it follows that $a \in N$.
This implies that $M \leq NM \leq N$, a contradiction. Hence, $N$ is the unique maximal normal abelian subgroup of $G$.

Let $x \in G$. Then, the group $\langle x, N \rangle$ is either abelian or metabelian with $x$ acting on $N$ by scalar multiplication. 
Furthermore, if $Z(G) \neq 1$, then $Z(G) \leq N$, and $x$ commutes with every element in $N$.
Now, it is clear that $(ii)$ and $(iii)$ hold.

Suppose that $Z(G)=1$ but $N \neq 1$. Since $N$ is normal in $G$, the centralizer of $N$ is also normal in $G$. Moreover, as $Z(C_G(N))$ is characteristic in $C_G(N)$, it
follows that $Z(C_G(N))$ is a normal abelian subgroup of $G$. Hence,  $Z(C_G(N))  \leq N$, and as the reverse inclusion is obvious, we get $Z(C_G(N)) = N$.

It follows from Corollary~\ref{virtually-abelian} $(ii)$ that $C_G(N)$ is isolated in $G$ (because every intersection of isolated subgroups is also isolated).
Hence, $G/C_G(N)$ is torsion free. We need to prove that $G/C_G(N)$ is pro-cyclic.

Fix $n \in N$, $n \neq 1$. It follows from $(ii)$ that every non-trivial element of $G/C_G(N)$ acts on $\langle n \rangle$ by non-trivial scalar multiplication. 
Hence, $G/C_G(N)$ embeds into $\textrm{Aut}(\langle x \rangle)$. Since $G/C_G(N)$ is torsion free, it follows that it is isomorphic to $\Z_p$.

\end{proof}

\section{Free pro-$p$ groups and Demushkin groups}
\label{Free and Demushkin}

Let $F$ be a free pro-$p$ group, and let $H$ be a subgroup of $F$. Then, $H$ is also free pro-$p$ and $H^{ab}$ is a free abelian pro-$p$ group.
It follows from Corollary~\ref{CR - criteria} $(ii)$ that $(H, \Phi(H), [H,H])$ is a hierarchical triple. 

\begin{thm}
\label{free}
Every free pro-$p$ group is strongly commutator-resistant.
\end{thm}

A pro-$p$ group $G$ is called a Demushkin group if it satisfies the following conditions:
\begin{itemize}
\item[(i)] $\textrm{dim} _{\mathbb{F}_p} H^1(G,\mathbb{F}_p) < \infty$,
\item[(ii)] $\textrm{dim} _{\mathbb{F}_p} H^2(G,\mathbb{F}_p)=1$, \textrm{ and }
\item[(iii)] the cup-product $H^1(G,\mathbb{F}_p)\times H^1(G,\mathbb{F}_p) \to H^2(G,\mathbb{F}_p) \cong \mathbb{F}_p$ is a non-degenerate bilinear form.
\end{itemize}

If $k$ is a $p$-adic number field containing a primitive $p$th root of unity and $k(p)$ is a maximal $p$-extension of $k$, then $\textrm{Gal}(k(p)/k)$ is a Demushkin group. 
Furthermore, the  pro-$p$ completion of any orientable surface group is also a Demushkin group. In fact, all Demushkin groups have many properties reminiscent of surface groups. For instance, 
every finite index subgroup $U$ of a Demushkin group $G$ is a Demushkin  group with $d(U) = |G:U|(d(G)-2)+2$, and every subgroup of infinite index is free pro-$p$. For a detailed exposition of the theory of Demushkin groups, see \cite{Se97} or  \cite{NeSchWi08}. 

Let $G$ be a Demushkin group. Since $\textrm{dim} _{\mathbb{F}_p} H^2(G,\mathbb{F}_p)=1$, it follows that $G$ is a one related pro-$p$ group. Hence, there is an epimorphism $\pi: F \to G$ where $F$ is a free pro-$p$ group of rank $d:=d(G)$ and $\textrm{ker }\pi$ is generated as a closed normal subgroup by one element $r \in \Phi(F)=F^p[F, F]$. It follows that either $G^{ab}\cong \mathbb{Z}_p^d$ or $G^{ab}\cong \mathbb{Z}/p^e\mathbb{Z} \times \mathbb{Z}_p^{d-1}$ for some $e \geq 1$; set $q:=p^e$ in the latter and $q:=0$ in the former case. Then, $d$ and $q$ are two invariants associated to  $G$.  (When we wish to emphasize the Demushkin group under consideration, we write $d(G)$ and $q(G)$ for the invariants of $G$.) 

Demushkin groups were classified by  Demushkin, Serre and Labute (\cite{Demushkin1}, \cite{Demushkin2}, \cite{Serre2}, and  \cite{Labute}). We summarise the classification in the following

\begin{thm}
\label{demushkin - classification}
Let $G$ be a Demushkin  group with invariants $d$ and $q$.
Then $G$ admits a presentation $G=\langle x_1,x_2, \ldots x_d \mid r \rangle$, where 
\begin{enumerate}[(i)]
\item if $q \neq 2$, then $d$ is even and 
$$r=x_1^q[x_1,x_2][x_3, x_4]\cdots [x_{d-1}, x_d];$$  
\item if $q=2$ and $d$ is even, then 
\[r=x_1^{2+\alpha}[x_1,x_2]x_3^{2^f}[x_3, x_4]\cdots [x_{d-1}, x_d]\] 
for some  $f= 2, 3, \ldots, \infty$ $( 2^{f}=0 \text{ when } f=\infty)$ and $\alpha \in 4\mathbb{Z}_2$;
\item if $q=2$ and $d$ is odd, then 
$$r=x_1^2x_2^{2^f}[x_2,x_3]\cdots [x_{d-1}, x_d]$$ 
for some $f= 2, 3, \ldots, \infty$.
\end{enumerate}
\end{thm}

\medskip

We first consider Demushkin groups with $q \neq 2$. (As is often the case, ${q=2}$ takes more effort.)

\begin{thm}
\label{q>2}
Let $G$ be a Demushkin pro-$p$ group.
Then, the following assertions hold: 
\begin{enumerate}[(i)]
\item If $q(G) \neq p$, then $G$ is strongly commutator-resistant.
\item If $q(G)=p$ and $p$ is odd, then $G$ is strongly Frattini-resistant, but not commutator-resistant. 
\end{enumerate} 
\end{thm}
\begin{proof}
$(i)$ Suppose that $q(G) \neq p$, and let $H$ be a subgroup of $G$. If $|G:H| < \infty$, then $H$ is a Demushkin group with $q(H) \neq p$ (\cite[\S 3.Corollary]{Labute}), and it follows from Corollary~\ref{CR - criteria} $(i)$ that $(H, \Phi(H), [H,H])$ is a hierarchical triple;
otherwise, $H$ is a free pro-$p$ group, and $(H, \Phi(H), [H,H])$ is a hierarchical triple by Theorem~\ref{free}. 

\medskip

$(ii)$ Suppose that $p$ is odd and $q(G)=p$. Since $G^{ab} \cong \mathbb{Z}/p\mathbb{Z} \times \mathbb{Z}_p^{d-1}$, it follows from Corollary~\ref{CR - criteria} $(i)$ that the triple $(G, \Phi(G), [G,G])$ is not hierarchical.
Therefore, $G$ is not commutator-resistant.

Let $U$ be an open subgroup of $G$. Then, $U$ is a Demushkin group and there exist elements $x_1, \ldots, x_d \in U$ such that $U=\langle x_1, \ldots x_d \mid x_1^q[x_1,x_2][x_3, x_4]\cdots [x_{d-1}, x_d] \rangle$, where $q=q(U)=p^e$ for some $e \geq 1$.
Consider the hereditarily uniform pro-$p$ group 
$$K:=\langle z_1,z_2, \ldots, z_d \mid [z_i, z_j]=1  \text{ and }  z_1^{-1}z_iz_1=z_i^{1-q} \text{ for all } 2 \leq i, j \leq d \rangle \cong \Z_p \ltimes \Z_p^{d-1}.$$
The assignment $x_1 \mapsto z_2$, $x_2 \mapsto z_1$ and $x_i \mapsto z_i$  for $3 \leq i \leq d$ defines a Frattini-cover $U \to K$. 
Since  $K$ is Frattini-resistant (Proposition~\ref{p-adic analytic Frattini-resistant}), it follows from Proposition~\ref{F-R criterion} that $G$ is strongly Frattini-resistant.
\end{proof}

Before we turn to the case ${q=2}$, we prove two auxiliary results. 

\begin{lem}
\label{normal-Frattini}
Let $G$ be a pro-$p$ group that contains two distinct open subgroups with the same Frattini. 
Then, there are two distinct open subgroups $M$ and $N$ of $G$ such that $\Phi(M)=\Phi(N)$ and $M, N \unlhd \langle M, N \rangle$.
\end{lem}
\begin{proof}
Let $X$ be the set of all pairs $(U, V)$ of distinct open subgroups of $G$ with $\Phi(U)=\Phi(V)$, and let
$$r:=\min\{||G:U|-|G:V|| \mid (U,V) \in X\}.$$
Set $X_r:=\{(U, V) \in X \mid r=||G:U|-|G:V||\}$, and fix some $(U_0,V_0) \in X_r$.
Among all pairs $(U,V) \in X_r$ such that 
$|G:U|=|G:U_0|$ and $|G:V|=|G:V_0|$, choose one, say $(M,N)$, that maximizes the index $|G: \langle U, V \rangle|$.
We claim that both $M$ and $N$ are normal in $\langle M, N \rangle$. Suppose to the contrary that $M$ is not normal in $\langle M, N \rangle$. Then, there is $x \in N$ such that 
$M \neq M^x$. Moreover,
$$\Phi(M^x)=\Phi(M)^x=\Phi(N)^x=\Phi(N)=\Phi(M),$$
so $(M,M^x) \in X$. Since $|G:M|=|G:M^x|$, it follows that $r=0$, and thus $|G:M|=|G:N|$. Furthermore, by the choice of $(M,N)$, we have that
$|G:\langle M, M^x \rangle| \leq |G:\langle M, N \rangle|$. This, in turn, implies that $\langle M, M^x \rangle = \langle M, N \rangle$. However, this is only possible if $M=\langle M, N \rangle$, 
which yields a contradiction with $M \neq N$ and $|G:M|=|G:N|$.  
\end{proof}

\begin{lem}
\label{log}
Let $G$ be a pro-$2$ group, and let $M$ and $N$ be normal open subgroups of $G$ such that $G=MN$ and $\Phi(M)=\Phi(N)$.
Then 
$$\log_2(|\Phi(G):\Phi(M)|) \leq \log_2(|M:M\cap N|)\log_2(|N: M \cap N|).$$
\end{lem}
\begin{proof}
For an arbitrary element $g=mn$ ($m \in M, n \in N$) of $G$, we have
$$g^2=m^2n^2[n,m][[n,m],n].$$
Clearly, $m^2, n^2 \in \Phi(M)$, and as $[n,m] \in N$, we also get $[[n,m],n] \in \Phi(M)$. 
Therefore, $g^2 \in [M,N]\Phi(M) \leq \Phi(G)$. Since the squares of elements of $G$ generate  $\Phi(G)$, it follows that $\Phi(G)=[M,N]\Phi(M)$.

From $\Phi(M)=\Phi(N) \leq M \cap N$, we deduce that $M/M \cap N$ and $N/M \cap N$ are elementary abelian $2$-groups. In addition, 
$\Phi(G)/\Phi(M)$ is also elementary abelian, and it is easily seen that 
$$M/(M \cap N) \times N/(M \cap N) \to \Phi(G)/\Phi(M), (m (M\cap N) , n (M \cap N)) \mapsto [m,n]\Phi(M),$$  
is a well-defined bilinear map. Since  $\Phi(G)=[M,N]\Phi(M)$, the induced linear transformation
$M/(M \cap N) \otimes_{\mathbb{F}_2} N/(M \cap N) \to \Phi(G)/\Phi(M)$ is surjective. Hence,
\begin{align*}
&\log_2(|\Phi(G):\Phi(M)|)=\dim_{\mathbb{F}_2}\Phi(G)/\Phi(M) \leq \dim_{\mathbb{F}_2}M/(M \cap N) \otimes_{\mathbb{F}_2} N/(M \cap N)\\
&=\dim_{\mathbb{F}_2}M/(M\cap N)\dim_{\mathbb{F}_2}N/(M\cap N)=\log_2(|M:M\cap N|)\log_2(|N: M \cap N|).
\end{align*}
\end{proof}

\begin{thm}
\label{Demuskin-2}
Let $G$ be a Demushkin group with $q(G)=2$. Then, the following assertions hold:
\begin{enumerate}[(i)]
\item If $d(G)=2$, then $G$ is not Frattini-injective.
\item If $d(G)>2$, then $G$ is Frattini-injective, but not Frattini-resistant.   
\end{enumerate} 
\end{thm}
\begin{proof}
$(i)$ If $d(G)=2$, then $G=\langle x_1, x_2 \mid x_1^{2+\alpha}[x_1,x_2]\rangle$ for some  $\alpha \in 4\Z_2$. 
It is easily seen that $G$ is $2$-adic analytic, however, it is not isomorphic to any of the pro-$2$ groups listed in Theorem~\ref{main} (for instance, it is obvious that $G$ is not powerful).
Therefore, $G$ is not Frattini-injective.

\medskip

$(ii)$  Suppose that $d(G)>2$, and assume that there are two distinct finitely generated subgroups $H$ and $K$ of $G$ such that $\Phi(H)=\Phi(K)$.
Since a subgroup of $G$ is open if and only if it has open Frattini, $H$ and $K$ are either both open or they both have infinite index in $G$.
Furthermore, in case that $H$ and $K$ are open, they necessarily have the same index in $G$ (since for $U$ open, $|G:\Phi(U)|$ is a strictly increasing function of $|G:U|$).  

First suppose that $H$ and $K$ are open. By Lemma~\ref{normal-Frattini}, we may further assume that $H$ and $K$ are both normal in $L:=\langle H,K \rangle$.
Moreover, it follows from Lemma~\ref{log} that 
$$\log_2(|\Phi(L):\Phi(H)|) \leq \log_2(|H:H\cap K|)\log_2(|K: H \cap K|).$$
Hence,
\begin{align*}
\log_2(|L:\Phi(H)|)&= \log_2(|L:\Phi(L)|)+\log_2(|\Phi(L): \Phi(H)|)\\
&\leq d(L)+ \log_2(|H:H\cap K|)\log_2(|K: H \cap K|).
\end{align*}
On the other hand,
$$\log_2(|L:\Phi(H)|)= \log_2(|L:H|)+\log_2(|H: \Phi(H)|)=\log_2(|L:H|)+d(H).$$
Therefore,
$$\log_2(|L:H|)+d(H) \leq d(L)+ \log_2(|H:H\cap K|)\log_2(|K: H \cap K|).$$
Since $L$ is a Demuskin group, we have $d(H)=|L:H|(d(L)-2) +2$; in addition, $|H:H\cap K| \leq |L:K|=|L:H|$ and $|K:H\cap K| \leq |L:H|$.
Thus
$$\log_2(|L:H|)+|L:H|(d(L)-2)+2 \leq d(L)+ \log_2(|L:H|)^2.$$
Since $d(L) \geq 3$, it follows that
$$|L:H|-1\leq(|L:H|-1)(d(L)-2) \leq  \log_2(|L:H|)^2 - \log_2(|L:H|).$$
Setting $n:=\log_2(|L:H|)$, we obtain 
$$2^n \leq n^2 - n + 1.$$
However, this inequality holds only for $n=0$; so, we must have $L=H=K$, a contradiction. 

Now assume that $H$ and $K$ have infinite index in $G$. By Theorem~\ref{free}, $\langle H, K \rangle$ can not be free, so it must be open in $G$.
Moreover, since $\Phi(H)$ has finite index in both $H$ and $K$, 
it follows from the Greenberg property of Demushkin groups (\cite[Theorem~A]{SnoZa16}) that $\Phi(H)$ also has finite index in $\langle H, K \rangle$.
This implies that $H$ and $K$ are open in $G$, a contradiction.

\medskip

It remains to prove that $G$ is not Frattini-resistant. Suppose first that $d(G)$ is odd. Then
$G=\langle x_1, \ldots, x_d \mid x_1^2x_2^{2^f}[x_2,x_3]\cdots [x_{d-1}, x_d] \rangle$
for some $f= 2, 3, \ldots, \infty$. Let $H$ be the subgroup of $G$ generated by $x_2, x_3, \ldots, x_d$. Then $x_1 \notin H$, but
$x_1^2=[x_d, x_{d-1}] \cdots [x_3,x_2]x_2^{-2^f} \in \Phi(H)$. Therefore, $G$ is not Frattini-resistant. 

Now assume that $d(G)$ is even. Then 
$$G=\langle x_1, \ldots, x_d \mid x_1^{2+\alpha}[x_1,x_2]x_3^{2^f}[x_3, x_4]\cdots [x_{d-1}, x_d] \rangle$$
for some  $f= 2, 3, \ldots, \infty$  and $\alpha \in 4\mathbb{Z}_2$. Let $M$ be a maximal subgroup of $G$ that contains the elements
$x_2, x_3, \ldots, x_d$; then $x_1x_2 \notin M$. We may write the square of $x_1x_2$ as 
$$(x_1x_2)^2=x_1^2x_2^2[x_2, x_1][[x_2,x_1],x_2]=(x_1^{2}x_2x_1^{-2})^2(x_1^2[x_1,x_2])[x_2, x_1]^2[[x_2,x_1],x_2].$$
Clearly, $(x_1^{2}x_2x_1^{-2})^2 \in \Phi(M)$, and as $[x_2, x_1] \in M$, we also have $[x_2, x_1]^2 \in \Phi(M)$ and  $[[x_2,x_1],x_2] \in \Phi(M)$. 
Furthermore, $x_1^{\alpha} \in \Phi(M)$ (since $\alpha \in 4\Z_2$), and it follows from the relation of $G$ that
$$x_1^2[x_1, x_2]=x_1^{-\alpha }[x_d,x_{d-1}] \cdots [x_4, x_3]x_3^{-2^f} \in \Phi(M).$$
Therefore, $(x_1x_2)^2 \in \Phi(M)$, which proves that $G$ is not Frattini-resistant.

\end{proof}

\section{Maximal pro-$p$ Galois groups}
\label{Maximal pro-$p$ Galois groups}

Recall that we denote by $G_k=\textrm{Gal}(k_s/k)$ and ${G_k(p) =  \textrm{Gal}(k(p)/k)}$ the absolute Galois group and the maximal pro-$p$ Galois group of a field $k$, respectively.

\begin{thm}
\label{Galois-Frattini-resistant}
Let $p$ be an odd prime, and let $k$ be a field that contains a primitive $p$th root of unity. Then $G_k(p)$ is a strongly Frattini-resistant pro-$p$ group. 
\end{thm}
\begin{proof}
By Corollary~\ref{open-maximal}, it suffices to prove that for every subgroup $H$ of $G_k(p)$ and every maximal subgroup $M$ of $H$, the triple $(H, M, \Phi(M))$ is hierarchical.
Let $F$, $K$ and $L$ be the fixed fields of $H$, $M$ and $\Phi(M)$, respectively. Then $K=F(\sqrt[p]{a})$ for some $a \in F^{\times}$ (since $F$ contains a primitive $p$th root of unity). 
Let $b \in k(p)$ be a root of the polynomial $X^p-\sqrt[p]{a}$; then $[K(b):K]$ divides $p$, and hence $b \in L$.

Let $\sigma \in H \setminus M$. We claim that $\sigma^p$ does not fix $b$, and consequently $\sigma^p \notin \Phi(M)$.
 Since the roots of the polynomial $X^{p^2}-a$ are $b\zeta_{p^2}^i$  ($0 \leq i \leq p^2-1$), where 
$\zeta_{p^2}$ is a primitive $p^2$th root of unity, we have $\sigma(b)=b\zeta_{p^2}^s$ for some $0 \leq s \leq p^2-1$. 
Furthermore, $s$ is relatively prime to $p$, since $p \mid s$ would imply  
\[\sigma(\sqrt[p]{a})=\sigma(b^p)=\sigma(b)^p=b^p\zeta_{p^2}^{sp}=\sqrt[p]{a},\]
which yields a contradiction with $\sigma \notin M$.

Consider the action of $\sigma$ on the group $\mu_{p^2}=\langle \zeta_{p^2} \rangle$ of $p^2$th roots of unity. Since $\zeta_{p^2}^p \in k$, we have $\sigma(\zeta_{p^2}^p)=\zeta_{p^2}^p$ , and thus 
$\sigma(\zeta_{p^2})=\zeta_{p^2}^t$ for some $1 \leq t \leq p^2-1$ with $t  \equiv 1 (\textrm{mod } p)$, i.e., $t=1+lp$ for some $0 \leq l \leq p-1$. 
Moreover, by a simple calculation, we obtain
\[\sigma^p(b)=b\zeta_{p^2}^{s+ts+ \ldots +t^{p-1}s}.\] 
We need to prove that $p^2$ does not divide $s+ts+ \ldots +t^{p-1}s=s(1+t+ \ldots+t^{p-1})$. Since $p$ does not divide $s$, it suffices to show that
$p^2$ does not divide $1+t+ \ldots+t^{p-1}$. This is clearly the case if $t=1$, so we may assume that $t=1+lp$ for some $1 \leq l \leq p-1$. Then
\[1+t+\ldots +t^{p-1}=\frac{t^p-1}{t-1}=\frac{\sum_{i=1}^p\binom{p}{i}(lp)^i}{lp}=p+p^2[\frac{p-1}{2}l+\sum_{i=3}^p\binom{p}{i}p^{i-3}l^{i-1}].\]
Since $p$ is odd, it follows that $p^2$ does not divide $1+t+ \ldots+t^{p-1}$. 
\end{proof}

Clearly, the assumption that $p$ is an odd prime is essential in Theorem~\ref{Galois-Frattini-resistant} (for instance, $\mathbb{C}/\mathbb{R}$ is a maximal $2$-extension with Galois group cyclic of order two, which is obviously not  Frattini-injective).
However, the condition on the prime is used only in the last line of the proof (when $t \neq 1$). Hence, under the stronger assumption that $k$ contains a primitive $p^2$th root of unity, the theorem holds also for $p=2$.
In fact, in that case we can say more. 

\begin{thm}  \label{p^2 root}
Let $k$ be a field that contains a primitive $p^2$th root of unity. Then $G_k(p)$ is strongly commutator-resistant. 
\end{thm}
\begin{proof}
Let $H$ be a subgroup of $G_k(p)$ with fixed field $F$, and let $\phi:H \to \Z_p/p\Z_p$ be an epimorphism. Denote by $K$ the fixed field of $\textrm{ker } \phi$.
Then $K=F(\sqrt[p]{a})$ for some $a \in F^{\times}$, and $\phi$ induces an isomorphism $\bar{\phi}:H/\textrm{ker }\phi \cong \textrm{Gal}(K/F) \to \Z_p /p\Z_p$.   
Since $F$ contains a $p^2$th root of unity, there is a field $L$ containing $K$ such that
the extension $L/F$ is cyclic of degree $p^2$ (take $L:=F(\sqrt[p^2]{a})$). By composing the restriction homomorphism $H \to \textrm{Gal}(L/F)$ with a suitable isomoprhism 
from $\textrm{Gal}(L/F)$ to $\Z_p/p^2\Z_p$, we obtain an epimorphism $\psi: H \to \Z_p/p^2\Z_p$ such that $\pi \circ \psi = \phi$, where $\pi: \Z_p/p^2\Z_p \to \Z_p/p\Z_p$ is the natural projection.
It follows from Proposition~\ref{extension} that $G_k(p)$ is strongly commutator resistant.
\end{proof}

For a field $F$ containing a primitive $p$th root of unity, we denote by $F(\sqrt[p]{F^{\times}})$ the maximal $p$-Kummer extension of $F$.
In terms of field extensions, the Frattini-resistance of $G_k(p)$ takes the following form.

\begin{cor}
 Let $k$ be a field that contains a primitive $p$th root of unity.  If $p=2$, in addition, assume that $ \sqrt{-1}  \in k$. Then, for all intermediate fields $F$ and $K$ of the extension $k(p)/k$, 
\[F \subseteq K \iff  F(\sqrt[p]{F^{\times}}) \subseteq K(\sqrt[p]{K^{\times}})\]
Moreover, $F/k$ is Galois if and only if $F(\sqrt[p]{F^{\times}})/k$ is Galois. 
\end{cor}
\begin{proof}
Note that if $H$ is a subgroup of $G_k(p)$ with fixed field $F$, then $F(\sqrt[p]{F^{\times}})$ is the fixed field of $\Phi(H)$. Hence, by Galois correspondence, the corollary follows from Theorem~\ref{Galois-Frattini-resistant}, Theorem~\ref{p^2 root}
and Lemma~\ref{normal}.   
\end{proof}

As an immediate consequence of the torsion-freeness of Frattini-injective pro-$p$ groups, we obtain Becker’s restriction on the finite subgroups of $G_k(p)$ (\cite{Becker}).  

\begin{cor}
Let $k$ be a field that contains a primitive $p$th root of unity. 
\begin{enumerate}[(i)]
\item If $p$ is odd or $\sqrt{-1} \in k$, then $G_k(p)$ is torsion free and has the unique extraction of roots property. Furthermore, for every finitely generated subgroup $H$ of $G_k(p)$ and every open subgroup $U$ of $H$, 
we have $d(U) \geq d(H)$.
\item If $p=2$, then every non-trivial finite subgroup of $G_k(p)$ is cyclic of order two.
\end{enumerate}
\end{cor}
\begin{proof}
$(i)$ follows from the results of Section~\ref{basic results}. For the proof of $(ii)$, suppose that $H$ is a non-trivial finite subgroups of $G_k(2)$ with fixed field $F$.
It follows from Theorem~\ref{p^2 root} that $k(2)=F(\sqrt{-1})$. Hence, $H$ is a cyclic group of order two.
\end{proof}

Before we turn to the more general context of $1$-smooth cyclotomic pro-$p$ pairs, we give the proof of Theorem~\ref{absolute Galois}.

\begin{proof}[Proof of Theorem~\ref{absolute Galois}]
First suppose that $p$ is an odd prime. Since every pro-$p$ subgroup of the absolute Galois group $G_k$ is contained in a $p$-Sylow subgroup, it suffices to prove that every $p$-Sylow subgroup $P$ of $G_k$ is strongly Frattini-resistant.
If $k$ is of characteristic $p$, then $P$ is a free pro-$p$ group by \cite[Theorem~6.1.4]{NeSchWi08}. Hence, $P$ is strongly Frattini-resistant by Theorem~\ref{free}. 
If $k$ is of characteristic different than $p$, then the fixed field $F$ of $P$ contains a primitive $p$th root of unity and $P=G_F(p)$; in this case, the claim follows from Theorem~\ref{Galois-Frattini-resistant}.
For $p=2$, the result follows from Theorem~\ref{p^2 root}. 
\end{proof}

\subsection{$1$-smooth pro-$p$ groups}

Following \cite{Ef} and \cite{EfQu19}, we call a pair ${\mathcal{G}=(G, \theta)}$ consisting of a pro-$p$ group $G$ and a homomorphism $\theta:G \to 1+p\Z_p$ (where $1+p\Z_p$ is the group of $1$-units of $\Z_p$)  a \emph{cyclotomic pro-$p$ pair}.

Given a cyclotomic pro-$p$ pair $\mathcal{G}=(G, \theta)$, let $\Z_p(1)$ be the $G$-module with underlying abelian group $\Z_p$ and $G$ action defined by
$g \cdot v = \theta(g)v$ for all $g \in G$ and $v \in \Z_p$. The pair $\mathcal{G}$ is said to be \emph{$1$-smooth} (or $1$-cyclotomic; see  \cite{CleFlo}, \cite{EfQu19}, \cite{QuWe18}, \cite{Qu20a} and \cite{Qu20b}), 
if for every open subgroup $U$ of $G$ and every $n \in \mathbb{N}$, the quotient map $\Z_p(1)/p^n\Z_p(1) \to \Z_p(1)/p\Z_p(1)$ (considered as a $U$-module homomorphism) 
induces an epimorphism
\[H^1(U, \Z_p(1)/p^n\Z_p(1)) \to H^1(U, \Z_p(1)/p\Z_p(1)).\] 

Let $k$ be a field containing a primitive $p$th root of unity, and let $\mu_{p^\infty}$ be the group of all roots of unity in $k(p)$ of order a power of $p$. 
The cyclotomic pro-$p$ character $\theta_{k,p}:G_k(p) \to 1+p\Z_p$ is defined by $\sigma(\zeta)=\zeta^{\theta_{k,p}(\sigma)}$ for all $\sigma \in G_k(p)$ and $\zeta \in \mu_{p^{\infty}}$.
One can readily prove (using Hilbert $90$) that $(G_k(p), \theta_{k,p})$ is a $1$-smooth cyclotomic pro-$p$ pair. Moreover, if $p=2$ and $\sqrt{-1} \in k$, then $\textrm{Im}(\theta_{k,2}) \leq 1+4\Z_2$.

\begin{lem}
\label{quasi-smooth}
Let $G$ be a pro-$p$ group. Suppose that the abelian group $\Z_p/p^2\Z_p$ can be endowed with a structure of a (topological) $G$-module in such a way that for every open subgroup $U$ of $G$,
the quotient homomorphism  $\pi:\Z_p/p^2\Z_p \to \Z_p/p\Z_p$  induces an epimorphism  
\[\pi^*:H^1(U, \Z_p/p^2\Z_p) \to H^1(U, \Z_p/p\Z_p).\]
Then, the following assertions hold:
\begin{enumerate}[(i)]
\item If $\Z_p/p^2\Z_p$ is the trivial $G$-module, then $G$ is strongly commutator-resistant.
\item Some maximal subgroup of $G$ is strongly commutator-resistant.
\item If $p$ is an odd prime, then $G$ is strongly Frattini-resistant. 
\end{enumerate}
\end{lem}
\begin{proof}
If $G$ acts trivially on $\Z_p/p^2\Z_p$, then every open subgroup of $G$ satisfies the extension of homomorphisms property of Proposition~\ref{extension}. Hence, $(i)$ follows from Corollary~\ref{commutator open}.
In general (for an arbitrary action), we obtain a continuous homomophism from $G$ to $\textrm{Aut}(\Z_p/p^2\Z_p)$ with kernel $M$ such that $|G:M| \leq p$.
Now $M$ acts trivially on $\Z_p/p^2\Z_p$, and it follows from $(i)$ that $M$ is strongly commutator-resistant, whence $(ii)$.

For the proof of $(iii)$, let $U$ be an open subgroup of $G$, and let $M$ be a maximal subgroup of $U$.
By Proposition~\ref{open-resistant}, it suffices to prove that $(U, M, \Phi(M))$ is a hierarchical triple. Upon identifying $U/M$ with $\Z_p/p\Z_p$, we may consider the natural projection ${\phi:U \to U/M}$ as an element of  
$H^1(U, \Z_p/p\Z_p)=\textrm{Hom}(U, \Z_p/p\Z_p)$. The surjectivity of the homomoprhism $H^1(U, \Z_p/p^2\Z_p) \to H^1(U, \Z_p/p\Z_p)$ implies the existence of a derivation ($1$-cocycle)
$d:U \to \Z_p/p^2\Z_p$ such that $\pi \circ d=\phi$. For every $x \in M$, we have $\pi(d(x))=\phi(x)=0$; so, $d(M) \leq p\Z_p/p^2\Z_p$. Since $p\Z_p/p^2\Z_p$ is necessarily a trivial $G$-module, the restriction of $d$ to $M$
is a homomoprhism. It follows that $d(\Phi(M))=0$.

Let $x \in U \setminus M$. Then $\phi(x) \neq 0$, and thus $d(x) \notin p\Z_p/p^2\Z_p$. Now $x \cdot d(x)= \alpha d(x)$ for some $\alpha = 1+ lp$ with $0 \leq l \leq p-1$, and 
\[d(x^p)=(1+\alpha+ \ldots + \alpha^{p-1})d(x).\]
If $\alpha=1$, then $d(x^p)=pd(x) \neq 0$, and thus $x^p \notin \Phi(M)$; otherwise 
$$1+\alpha+ \ldots + \alpha^{p-1}=\frac{\alpha^p-1}{\alpha -1}=\frac{\sum_{i=1}^p\binom{p}{i}(lp)^i}{lp}=p+p^2[\frac{p-1}{2}l+\sum_{i=3}^p\binom{p}{i}p^{i-3}l^{i-1}].$$
Since $p$ is assumed to be an odd prime, it follows that $p^2$ does not divide $\displaystyle{\frac{\alpha^p-1}{\alpha -1}}$. Therefore, $d(x^p) \neq 0$ and $x^p \notin \Phi(M)$.    
\end{proof}

\medskip

\begin{proof}[Proof of Theorem~\ref{1-smooth}]
For $p$ odd, this follows from Lemma~\ref{quasi-smooth} $(iii)$. If $p=2$ and $\textrm{Im}(\theta) \leq 1+4\Z_2$, then $G$ acts trivially on $\Z_2(1)/4\Z_2(1)$ and by Lemma~\ref{quasi-smooth}, $G$ is commutator-resistant. 
\end{proof}

The following two corollaries, in particular, subsume Theorem~\ref{Galois-analytic} and Theorem~\ref{Galois-solvable}. They were recently proved by Quadrelli  \cite{Qu20a}, \cite{Qu20b}.

\begin{cor}
Let $G$ be a $p$-adic analytic pro-$p$ group. Then, there exists a homomorphism $\theta:G \to 1+p\Z_p$ (with $\mathrm{Im}(\theta) \leq 1+4\Z_2$ if $p=2$) such that $(G, \theta)$ is a $1$-smooth cyclotomic pro-$p$ pair if and only if $G$ is one of the groups listed in Theorem~\ref{main}. 
\end{cor}
\begin{proof}
This follows from Theorem~\ref{1-smooth} and the well-known fact that the groups listed in Theorem~\ref{main} can be realized as maximal pro-$p$ Galois groups.
\end{proof}

\begin{cor}
Let $G$ be a solvable pro-$p$ group. Then, there exists a homomorphism $\theta:G \to 1+p\Z_p$  (with $\mathrm{Im}(\theta) \leq 1+4\Z_2$ if $p=2$)  such that $(G, \theta)$ is a $1$-smooth cyclotomic pro-$p$ pair if and only if 
$G$ is free abelian or it is a semidirect product $\langle x \rangle \ltimes A$, where $\langle x \rangle \cong \Z_p$, $A$ is a free abelian pro-$p$ group and 
$x$ acts on $A$ as scalar multiplication by $1+p^s$ with $s \geq 1$ if $p$ is odd, and $s \geq 2$ if $p=2$.
\end{cor}
\begin{proof}
This follows from Theorem~\ref{1-smooth} and Theorem~\ref{solvable}.
\end{proof}

By Theorem~\ref{max-normal}, for a field $k$ containing a primitive $p$th root of unity (and also $\sqrt{-1} \in k$ if $p=2$), $G_k(p)$ contains a unique normal abelian subgroup $N$.
Moreover, it was proved in \cite{QuWe18} that
$$N=\{h \in \ker \theta_{k,p} \mid ghg^{-1}=h^{\theta_{k,p}(g)} \text{ for all } g \in G_k(p)\},$$
where $\theta_{k,p}$ is the cyclotomic pro-$p$ character defined above. 

Given a pro-$p$ group $G$ and a homomorphism $\theta:G \to 1+p\Z_p$, let
$$Z_{\theta}(G):=\{h \in \ker \theta \mid ghg^{-1}=h^{\theta(g)} \text{ for all } g \in G\}.$$
Note that $Z_{\theta}(G)$ is a normal abelian subgroup of $G$. 
\begin{pro}
\label{theta-center}
Let $G$ be a Frattini-injective pro-$p$ group containing a non-trivial normal abelian subgroup.
The following assertions hold:
\begin{enumerate}[(i)]
\item There exists a unique homomorphism ${\theta: G \to 1+p\Z_p}$ such that 
$Z_{\theta}(G)$ is the (unique) maximal abelian normal subgroup of $G$.
\item If $G$ is not pro-cyclic and for some $\psi: G \to 1+p\Z_p$, $\mathcal{G}=(G, \psi)$ is a $1$-smooth cyclotomic pro-$p$ pair (with $\mathrm{Im}(\psi) \leq 1+4\Z_2$ if $p=2$), then $\psi=\theta$.
\end{enumerate}
\end{pro}
\begin{proof}
$(i)$ The uniqueness part is obvious. Let $N$ be the unique maximal normal abelian subgroup of $G$, whose existence is guaranteed by Theorem~\ref{max-normal}.
We define ${\theta: G \to 1+p\Z_p}$ as follows: if $N=Z(G)$, then take $\theta$ to be the trivial homomorphism (i.e., $\theta(g)=1$ for all $g \in G$);
otherwise, by Theorem~\ref{max-normal} $(iv)$, $G=\langle x \rangle\ltimes C_G(N)$ for some suitable $x \in G$, and we let $\theta(C_G(N))=1$ and  $xax^{-1}=a^{\theta(x)}$ for any (and hence every) $a \in N$.
It readily follows from Theorem~\ref{max-normal} that indeed $N=\Z_{\theta}(G)$.

\medskip

$(ii)$ If $G=\langle x \rangle\ltimes A$ is metabelian (decomposed in a semidirect product as in Theorem~\ref{solvable}), then $A$ is the isolator of $[G,G]$, and consequently, ${\psi(A)=1}$.
Moreover, for every $n \in \mathbb{N}$, $a \in A$ and a derivation $d:G \to \Z_p(1)/p^n\Z_p(1)$, we have $\theta(x)d(a)=d(a^{\theta(x)})=d(xax^{-1})=d(a^{\psi(x)})=\psi(x)d(a)$. 
Therefore, the surjectivity of the cohomology maps $H^1(U, \Z_p(1)/p^n\Z_p(1)) \to H^1(U, \Z_p(1)/p\Z_p(1))$
implies that $\psi=\theta$. By a similar argument, if $G$ is abelian (but not pro-cyclic), $\psi$ must be the trivial homomorphism.

In general, denoting by $N$ the unique maximal normal subgroup of $G$, for an arbitrary element $x \in G \setminus N$, the group $H:=\langle x, N \rangle$ is metabelian and $(H, \psi_{\mid H})$ is a $1$-smooth cyclotomic pro-$p$ pair. It follows from what has been already proved that 
$\psi(N)=1$ and $\psi(x)=\theta(x)$. 
\end{proof}

\section{$p$-power-injective pro-$p$ groups}
\label{p-power injective}

\begin{definition}
 We say that a pro-$p$ group $G$ is \emph{strongly $p$-power-injective} 
 ({\emph{$p$-power-injective}}) if for all (finitely generated) subgroups $H$ and $K$ of $G$,
\[H^p = K^p \implies H = K. \]
We call a pro-$p$ group $G$ \emph{strongly $p$-power-resistant} (\emph{$p$-power-resistant}) if  for every  (finitely generated) subgroup $H$ of $G$, the triple $(G,H, H^p)$ is hierarchical. 
\end{definition}
 
In our opinion, these are concepts deserving careful investigation. 
However, in this brief final section, we do little more than record several statements that follow from (or could be proved in a similar manner as) the main results of this paper.
     
\begin{pro}
\label{embedding1}
A pro-$p$ group $G$ is strongly $p$-power-resistant ($p$-power-resistant) if and only if for all (finitely generated) subgroups $H$ and $K$ of $G$,
\[ H \leq K \iff  H^p \leq K^p \]
Consequently, every (strongly) $p$-power-resistant  pro-$p$ group is (strongly) $p$-power-injective. 
\end{pro}

\begin{pro}
\label{open-resistant1}
Let $G$ be a pro-$p$ group. If $(G, U, U^p)$ is a hierarchical triple for every open subgroup $U$ of $G$, then
$G$ is strongly $p$-power-resistant. In particular, a finitely generated $p$-power-resistant pro-$p$ group is strongly $p$-power-resistant.  
\end{pro}

\begin{pro}
\label{open-resistant1}
A (strongly) Frattini-resistant pro-$p$ group is (strongly) $p$-power-resistant.  
\end{pro}
\begin{proof}
This follows from the fact that $G^p \leq \Phi(G)$ for every pro-$p$ group $G$.
\end{proof}

For a pro-$2$ group $G$, we have $\Phi(G)=G^2$. Hence, $2$-power-injectivity ($2$-power-resistance)
is the same as Frattini-injectivity (Frattini-resistance).

\begin{cor}
\label{$p$-power-resistant-free}
\begin{enumerate}[(1)]
\item Every free pro-$p$ group is  strongly $p$-power-resistant. 
\item Let $G$ be a Demushkin pro-$p$ group. Then, the following assertions hold: 
\begin{enumerate}[(i)]
\item If $q(G) \neq p$, or $q(G)=p$ and $p$ is odd, then $G$ is strongly $p$-power-resistant.
\item If $q(G)=2$ and $d(G)>2$, then $G$ is $p$-power-injective, but not $p$-power-resistant.  
\item If $q(G)=2$ and $d(G)=2$, then $G$ is not $p$-power-injective.
\end{enumerate} 
\item Let $k$ be a field that contains a primitive $p$th root of unity.  If $p=2$, in addition, assume that $ \sqrt{-1}  \in k$. Then $G_k(p)$ is strongly $p$-power-resistant. 
\end{enumerate}
\end{cor}

In contrast to Frattini-resistance, $p$-adic analytic $p$-power-resistant pro-$p$ groups are ubiquitous. 

\begin{pro}
\label{$p$-power-resistant-Analytic}
Every torsion free $p$-adic analytic pro-$p$ group of dimension less than $p$  is (strongly) $p$-power-resistant.  
\end{pro}
\begin{proof}
Let $G$ be a  torsion free $p$-adic analytic pro-$p$ group of dimension less than $p$.  By \cite[Theorem~A]{GoKl09}, every closed subgroup of $G$ is saturable. Thus to every subgroup $H$ of $G$ we can associate a saturable $\mathbb{Z}_p$-Lie algebra $L_H$; moreover, $L_{H^p} = pL_H$.
Now let $H$ and $K$ be  subgroups of $G$ such that  $H^p \leq K^p$. Then 
\[ H^p \leq K^p \implies  pL_H \leq  pL_K \implies L_H \leq L_K  \implies H \leq K. \]
Hence, by Proposition \ref{embedding1}, $G$ is a (strongly) $p$-power-resistant pro-$p$ group.
\end{proof}

\begin{cor}
\label{$p$-power-resistant-Analytic-uncountable}
Suppose that $p \geq 5$. Then there are uncountably many pairwise non-commensurable $p$-power-resistant  $p$-adic analytic pro-$p$ groups.
\end{cor}
\begin{proof}
This follows from \cite[Theorem~1.1]{Sno16}.
\end{proof}

\section{Final Remarks}

In this final section, we formulate several problems that we hope will stimulate further research on Frattini-injective pro-$p$ groups.

\begin{problem}
Is every finitely generated Frattini-injective pro-$p$ group strongly Frattini-injective?
\end{problem}

Theorem~\ref{Demuskin-2} provides examples of Frattini-injective pro-$2$ groups that are not Frattini-resistant.
We do not know any such examples for $p$ odd.

\begin{problem}
For $p$ an odd prime, find examples of Frattini-injective pro-$p$ groups that are not Frattini-resistant, or prove that such groups do not exist.
\end{problem}

In \cite{Ware92}, Ware proved that for $p$ odd the maximal pro-$p$ Galois group of a field containing a primitive $p$th root of unity is either metabelian or it contains a non-abelian free pro-$p$ subgroup. 

\begin{problem}
Does every non-metabelian Frattini-injective (Frattini-resistant) \mbox{pro-$p$} group contain a non-abelian free pro-$p$ subgroup?
\end{problem}

Following \cite{Wu85}, we call a pro-$p$ group $G$ \emph{absolutely torsion-free} if $H^{ab}$ is torsion-free for every subgroup $H$ of $G$.
Let $\mathcal{G}=(G, \theta)$ be a $1$-smooth cyclotomic pro-$p$ pair (with $\mathrm{Im}(\theta) \leq 1+4\Z_2$ if $p=2$), and suppose that $G$ has a non-trivial center. 
It follows from Proposition~\ref{theta-center} that $\theta$ is the trivial homomorphism. Moreover, the proof of Proposition~\ref{extension} can be adapted to show that $G$ is absolutely torsion-free.

\begin{problem}
Is every Frattini-injective (Frattini-resistant) pro-$p$ group with non-trivial center absolutely torsion-free? 
\end{problem}

In what follows assume that $p$ is odd. 

\begin{problem}
Let $G$ be a strongly Frattini-resistant pro-$p$ group.
Does there necessarily exist a homomorphism $\theta: G \to 1+p\Z_p$ such that $(G, \theta)$ is a $1$-smooth cyclotomic pro-$p$ pair?
(Note that if $G$ contains a non-trivial abelian normal subgroup, then Corollary~\ref{theta-center} gives a description of the only possible candidate for $\theta$.)
\end{problem}

\begin{problem}
Is every Frattini-resistant pro-$p$ group Bloch-Kato?
\end{problem}

In \cite{Ef}, Efrat introduced the class  $C$ of cyclotomic pro-$p$ pairs of elementary type. This class consists of all finitely generated cyclotomic pro-$p$ pairs which can be constructed 
from $Z_p$ and Demushkin groups using free pro-$p$ products and certain semidirect products, known also as fibre products (cf.  \cite[§ 3]{Ef}; see also \cite[§ 7.5]{QuWe18}). Given a  field $k$  that contains a primitive p$th$ root of unity,  Efrat conjectured that if $G_k(p)$ is finitely generated, then the cyclotomic pro-$p$ pair 
$(G_k(p), \theta_{k,p})$ is of elementary type;  this is the so-called elementary type conjecture (cf. \cite{Ef95}, \cite{Ef97} and  \cite{Ef}; see also  \cite[§ 7.5]{QuWe18}). 

It follows from \cite[Theorem~1.4]{QuWe18} and Theorem~\ref{1-smooth} that if $(G, \theta)$ is a  cyclotomic pro-$p$ pair of elementary type, then $G$ is a strongly Frattini resistant pro-$p$ group. 

\begin{problem}
Is there a finitely generated (strongly) Frattini resistant pro-$p$ group which is not of elementary type.
\end{problem}

A negative answer to the above question would settle the elementary type conjecture. 
On the other hand, if there exist counter examples, then they will likely be pro-$p$ groups with exotic properties.

\medskip

\ackn{The first author acknowledges support from the Alexander von Humboldt Foundation, CAPES (grant 88881.145624/2017-01), CNPq and FAPERJ.}

\bibliographystyle{plain}

\end{document}